\documentclass[reqno,11pt]{amsart}

\usepackage{amsmath,amsthm,amssymb,epsfig}
\usepackage{tikz}
\usepackage{float}
\usepackage{mathtools}
\usepackage{subcaption}
\usepackage{algorithm}
\usepackage{algorithmic}
\usepackage{MnSymbol}
\usepackage{graphicx,xcolor}
\usepackage{comment}
\usepackage{tkz-graph}
\usetikzlibrary{shapes}
\usepackage{comment}
\usepackage[left=3.5cm,top=3.5cm,right=3.5cm,bottom=3.5cm]{geometry}

\newtheorem{theorem}{Theorem}
\newtheorem{lemma}[theorem]{Lemma}
\newtheorem{corollary}[theorem]{Corollary}
\newtheorem{definition}[theorem]{Definition}
\DeclareMathOperator{\prox}{prox}

\DeclareMathOperator{\diam}{diam}
\DeclareMathOperator{\rad}{rad}
\newcommand{\hh}{h}

\title{The $k$-visibility Localization Game}

\author[A.\ Bonato]{Anthony Bonato}
\author[T.G.\ Marbach]{Trent G.\ Marbach}
\author[J.\ Marcoux]{John Marcoux}
\author[JD Nir]{JD Nir}
\address[A1,A2,A3]{Toronto Metropolitan University, Toronto, Canada}
\address[A4]{Oakland University, Rochester, U.S.A.}

\email[A1]{(A1) abonato@torontomu.ca}
\email[A2]{(A2) trent.marbach@torontomu.ca}
\email[A3]{(A3) jmarcoux@torontomu.ca}
\email[A4]{(A4) jdnir@oakland.edu}

\begin{document}


\maketitle
\begin{abstract}
We study a variant of the Localization game in which the cops have limited visibility, along with the corresponding optimization parameter, the $k$-visibility localization number $\zeta_k$, where $k$ is a non-negative integer. We give bounds on $k$-visibility localization numbers related to domination, maximum degree, and isoperimetric inequalities. For all $k$, we give a family of trees with unbounded $\zeta_k$ values. Extending results known for the localization number, we show that for $k\geq 2$, every tree contains a subdivision with $\zeta_k = 1$. For many $n$, we give the exact value of $\zeta_k$ for the $n \times n$ Cartesian grid graphs, with the remaining cases being one of two values as long as $n$ is sufficiently large. These examples also illustrate that $\zeta_i \neq \zeta_j$ for all distinct choices of $i$ and $j.$ 
\end{abstract}

\section{Introduction}

Pursuit-evasion games, such as the Localization game and the Cops and Robber game, are combinatorial models for detecting or neutralizing an adversary’s activity on a graph. In such models, pursuers attempt to capture an evader loose on the vertices of a graph. The rules describing how the players move and how the adversary is captured depend on which variant is studied. Such games are motivated by foundational topics in computer science, discrete mathematics, and artificial intelligence, such as robotics and network security. For surveys of pursuit-evasion games, see the books \cite{bonato1,bonato2}; see Chapter~5 of \cite{bonato1} for more on the Localization game.

Among the many variants of the game of Cops and Robbers, one theme is to limit the visibility of the robber. For a nonnegative integer $k,$ in $k$-visibility Cops and Robbers, the robber is visible to the cops only when a cop is distance at most $k$. The case when $k=0$ has been studied~\cite{der1,der2,Tosik1985}, as has the case when $k=1$~\cite{bty1, bty2, bty3}, and a recent paper studied the cases $k\geq1$~\cite{kvis_CR_Clarke2020}. 

The Localization game was first introduced for one cop in~\cite{car,seager1}. The game in the present form was first considered in the paper \cite{car}, and subsequently studied in several papers such as~\cite{BBHMP,BHM,BHM1,BK,Bosek2018,nisse2,BDELM,seager2}. We consider a version of limited visibility Cops and Robbers in the setting of the Localization game, first introduced in \cite{BMMN}. Let $k$ be a non-negative integer. In the $k$-\emph{visibility Localization game}, two players are playing on a graph, with one player controlling a set of $m$ \emph{cops}, where $m$ is a positive integer, and the second controlling a single \emph{robber}. 
The game is played over a sequence of discrete time-steps; a \emph{round} of the game is a move by the cops and the subsequent move by the robber. The robber occupies a vertex of the graph, and when the robber is ready to move during a round, they may move to a neighboring vertex or remain on their current vertex. 
A move for the cops is a placement of cops on a set of vertices (note that the cops are not limited to moving to neighboring vertices). The players move on alternate time-steps, with the robber going first. In each round, the cops $C_1,C_2,\ldots ,C_m$ occupy a set of vertices $u_1, u_2, \dots , u_m$ and each cop sends out a \emph{cop probe} $d_i$, where $1\le i \le m$. 
If a cop $C_i$ is at distance $j$ from the vertex of the robber, where $0 \le j \le k,$ then $d_i=j$. 
In all other cases, the cop probe returns no information, and we set $d_i=\ast$. Hence, in each round, the cops determine a \emph{distance vector} $D=(d_1, d_2, \dots ,d_m)$ of cop probes. 
Relative to the cops' position, there may be more than one vertex $x$ with the same distance vector that the robber may occupy. 
We refer to such a vertex $x$ as a \emph{candidate of} $D$ or simply a \emph{candidate}. 
The cops win if they have a strategy to determine, after a finite number of rounds, a unique candidate, at which time we say that the cops {\em capture} the robber. We assume the robber is \emph{omniscient}, in the sense that they know the entire strategy for the cops. If the robber evades capture, then the robber wins. For a graph $G$, define the $k$-\emph{visibility localization number} of $G$, written $\zeta_k(G)$, to be the least positive integer $m$ for which $m$ cops have a winning strategy in the $k$-visibility Localization game. The standard Localization game is played similarly, except that each cop's probe returns $d_i$ as the distance between this cop and the robber. The localization number is the minimum number of cops required for this game, denoted $\zeta(G)$. 

For a graph $G$ of order $n$ and for all non-negative integers $k$, $\zeta(G) \le \zeta_k(G)$, as a winning cop strategy in the $k$-visibility Localization game will also be winning in the Localization game. Further, $\zeta_k(G) \le n - 1$. If $G$ has diameter at most $k+1$, then a probe of $\ast$ by a $k$-visibility cop can only represent a distance of $k+1$, and so $\zeta(G) = \zeta_{\diam(G)-1}(G)$. A recent work~\cite{BL22} introduced the so-called \emph{zero-visibility search game}, which is equivalent to $\zeta_k$ when $k=0.$ The case $k=1$ was studied in \cite{BMMN} and the $k$-visibility Localization was studied for random graphs in \cite{kvisr}.

The paper is organized as follows. In the next section, we consider various bounds for the $k$-visibility localization number, connecting it to domination, maximum degree, and isoperimetric parameters. We study a variant of the $k$-visibility localization number called the $k$-proximity number, which will be used to construct bounds for trees and grids. 
Section~3 focuses on trees, and for all $k$, we give a family of trees with unbounded $\zeta_k$ values. 
Section~4 shows that every tree contains a subdivision with $\zeta_k = 1$ when $k \geq 2$. 
This extends an analogous result about $1$-visibility localization numbers of trees \cite{car} to $k$-visibility localization numbers. Section~5 focuses on Cartesian grids, in which  Theorem~\ref{thm:main_grids_results} gives the exact value of $\zeta_k$ for such graphs in many cases, with the remaining cases being one of two values as long as $n$ is sufficiently large. This result also separates the parameters, showing that for all $i<j$, there exists a graph $G$ with $\zeta_{i}(G) \neq \zeta_{j}(G),$ settling an open problem stated in \cite{BMMN}. We finish with open problems. 

All graphs we consider are finite, undirected, reflexive, and do not contain parallel edges. The variable $k$ stands for a non-negative integer. We only consider connected graphs unless otherwise stated. The set of vertices that share an edge with $x$ is denoted $N(x)$, and we refer to vertices in $N(x)$ as \emph{neighbors} of $x$. Although our graphs are reflexive, we insist that $x\notin N(x).$ We define $N[x]=N(x) \cup \{x\}$. For a set $S$ of vertices, $N[S]=\bigcup_{u \in S} N[u]$. For a graph $G,$ let $\Delta(G)$ be the maximum degree of a vertex in $G$. For further background on graph theory, see \cite{west}.

\section{General Bounds}

We begin by discussing some bounds for $\zeta_k(G)$ in terms of other graph parameters. These results have been selected to build the reader's intuition for $\zeta_k(G)$ and establish some tools that will be used later. To do so, we must first introduce the $k$-\emph{proximity game} as a tool for analyzing $\zeta_k$. The $k$-proximity game has the same rules as the $k$-visibility Localization game, but rather than winning by locating the robber, the cops win by having any cop be within distance $k$ of the robber. We denote the minimum number of cops required to win the $k$-proximity game on a graph $G$ as $\prox_k(G)$. This is a generalization of the \emph{blind localization game} first considered in \cite{Bosek2018}, which was later reintroduced as the proximity game in \cite{BMMN}. The $k$-proximity game is also analogous to the $k$-\emph{visibility seeing cop-number} introduced in \cite{kvis_CR_Clarke2020}. We begin by showing in Theorem~\ref{thm:prox_bound} that while the $k$-proximity number and $k$-visibility localization numbers are not necessarily the same, there is an upper and lower bound on $\zeta_k$ in terms of $\prox_k$. Note that for this paper, $\Delta_k(G)$ is defined as the maximum number of vertices within distance $k$ of any vertex $u \in V(G)$.

\begin{theorem} \label{thm:prox_bound}
For a graph $G$ and $k \ge 1$, $$\prox_k(G) \le \zeta_k(G) \leq \Delta_k(G) \prox_k(G).$$
\end{theorem}

\begin{proof}
The lower bound holds since, given any strategy to identify the robber by cops playing the $k$-visibility Localization game, the same number of cops playing the $k$-proximity game can find the robber's position and probe it in the following round. The robber must be within distance $1$ of that vertex so, as $k\ge 1$, this probe will reveal the robber.

For the upper bound, we simulate the strategy of the $k$-proximity cops by using a $k$-visibility localization cop to probe each vertex visible by a $k$-proximity cop. 
As the cops succeed in the $k$-proximity game, some $k$-proximity cop must come within distance $k$ of the robber.
In the corresponding cop move for the $k$-visibility proximity game, this corresponds to a cop probing the robber's exact location. 
\end{proof}

Recall that a \emph{dominating set} $S$ in a graph $G$ has the property that for each vertex $x \notin S$, some vertex in $S$ is adjacent to $x.$ The domination number of $G$, written $\gamma(G)$, is the minimum cardinality of a dominating set. In \cite{BMMN}, the authors prove that if $G$ is $C_4$-free, $\zeta_1(G) \le \gamma(G) + \Delta(G)$. To generalize their result, consider the size of a minimal $j$-dominating set, written $\gamma_j(G)$, which is a set $S$ with the property that for each vertex $x \notin S$, there is some vertex in $S$ at a distance at most $j$ from $x$.

\begin{theorem} \label{thm:girth_bound}
For a graph $G$ with girth at least $2j+3$ and any $k \ge j$ we have that $$\zeta_k(G) \leq \gamma_j(G) + \Delta_j(G).$$
\end{theorem}

\begin{proof}
We play the $k$-visibility Localization game with $\gamma_j(G) + \Delta_j(G)$ cops, and show that the cops can capture the robber. 
Let $S$ be a $j$-dominating set of size $\gamma_j(G)$. In each round, use $\gamma_j(G)$ cops to probe each vertex in $S$. As $k \ge j$, the cops can determine some $s \in S$ such that the robber is within distance $j$ from $s$. On the second round, the remaining $\Delta_j(S)$ cops probe each vertex in the $j$-neighborhood of $s$. If the robber is not on one of the probed vertices, then they must have moved to some vertex at distance $j+1$ from $s$. Some cop in the $\Delta_j$ ball around $s$, say at vertex $t$, must therefore be at distance $1$ from the robber. The robber is captured on that vertex if there is a unique neighbor of $t$ at distance $j+1$ from $s$. Otherwise, suppose $y$ and $z$ are neighbors of $t$ at distance $j+1$ from $s$. As $S$ is a $j$-dominating set, there is $s' \in S$ at distance at most $j$ from $y$. 

Suppose, for contradiction, that $z$ was also within distance $j$ from $s'$. The closed walk of length at most $2j+2$ that follows the path from $s'$ to $y$ takes the edge from $y$ to $t$, then from $t$ to $z$, and then returns to $s'$, contains a cycle of length at most $2j+2$. This contradicts that $G$ has girth at least $2j+3$. We conclude that if $\mathcal{C}$ is the set of candidates for the robber's position (that is, the set of neighbors of $t$ at distance $j+1$ from $s$), for each $v \in \mathcal{C}$ there is a $s' \in S$ that distinguishes $v$, which implies that $v$ is at distance at most $j$ from $s'$ and every other vertex in $\mathcal{C}$ is at distance greater than $j$ from $s'$, and therefore, the robber is captured in the second round.\end{proof}

Note that for a typical graph $G$, $\gamma_j(G)$ is decreasing in $j$ while $\Delta_j(G)$ is increasing. For graphs with high girth, Theorem~\ref{thm:girth_bound} offers flexibility as one can choose the value of $j$ minimizing the sum of these two variables. In applications where more is known about the structure of $G$, the technique of ``guarding'' a $j$-dominating set and chasing down the robber between the protected vertices can often be refined further.

In general, as in the Localization game, the $k$-visibility localization number of a graph $G$ is not in general monotonic on subgraphs of $G$, even if the subgraph is induced; see Exercises 5.6 and 5.7 in \cite{bonato1} for an example in the Localization game. However, in the case $G$ is a tree, we do observe monotonicity.

\begin{lemma} \label{lem:subtreebound}
For a tree $T$ with subtree $T'$, $\zeta_k(T') \leq \zeta_k(T)$ and $\prox_k(T') \le \prox_k(T)$.
\end{lemma}

\begin{proof}
Let $f:V(T) \to V(T')$ map $f(v) = v$ for each $v \in V(T')$ and, for $u \in V(T) \setminus V(T')$, let $f$ map $u$ to the unique vertex in $V(T')$ minimizing the distance (in $T$) from $u$. To play either game on $T'$, the cops follow their winning strategy on $T$. Whenever their strategy calls for probing a vertex $u \in V(T) \setminus V(T')$, they instead probe on $f(u)$ and add the distance (in $T$) from $u$ to $f(u)$ to the distance they receive. If this distance is greater than $k$, then the cops treat the probe as though it had not returned a distance. 

Since for any $v$ in $V(T')$ and $u$ in $V(T) \setminus V(T')$ the unique shortest path from $u$ to $v$ must pass through $f(u)$, this strategy will yield the exact same information as the strategy on $T$, but with the additional information that the robber cannot be on any of the vertices in $V(T) \setminus V(T')$. Thus, the candidate sets formed in this strategy will be subsets of the candidate sets formed in the strategy on $T$, so as the strategy succeeds on $T$, it will also succeed on $T'$.
\end{proof}

There are other ways in which the $k$-visibility Localization game is different to the Localization game when the graph is known to be a tree. We will discuss more consequences in Section~\ref{sec:trees}, but we prove one additional result here.

\begin{lemma} \label{lem:proxplusone}
    For a tree $T$, we have that $\zeta_k(T) \leq \prox_k(T) + 1$.
\end{lemma}

\begin{proof}
Note that that when $k=0$, $\zeta_k(T) = \prox_k(T)$, and so the statement clearly holds. We therefore assume for the rest of the proof that $k \geq 1$ and describe a strategy to use $\prox_k(T)+1$ cops to win the $k$-visibility localization game.

Choose a vertex $r \in T$ and root $T$ at $r$. We use one cop, say $C_r$, to probe $r$ each round. Let $T_1, T_2,\ldots, T_m$ be the connected components of $T-r$. To avoid immediate capture at $r$, the robber must be on $T_i$ for some $i$. Notice that the robber is unable to leave $T_i$ without going through $r$, at which point $C_r$ would probe a distance of $0$ and capture the robber, so the robber must remain on $T_i$.

By Lemma~\ref{lem:subtreebound}, $\prox_k(T_i) \le \prox_k(T)$ for each $i$. The remaining $\prox_k(T)$ cops follow a winning $k$-proximity game strategy on $T_i$. If these cops complete the strategy without detecting the robber, then the cops know the robber is not located on $T_i$ and proceed to $T_{i+1}$. 
Therefore, suppose that the cops probe a vertex $v$ in $T_i$ and find the robber is within distance $k$ of $v$, say distance $d_1$. We wish to move $C_r$ onto the root of $T_i$ and repeat this process, but we must be careful that the robber is truly located in $T_i$.
Let $d_2$ be the distance that the cop $C_r$ probed to the robber, with either $d_2\leq k$ or $d_2=*$, and let $d_3$ be the distance between $v$ and $r$. 
If $d_2=0$, then the robber is captured by the cop $C_r$, so we may assume that $1 \leq d_2 \leq k$ or $d_2=\ast$. 

Suppose the robber is not on $T_i$. We then have that $d_2 \neq *$ since $d_2$ is smaller than $d_1$ as $r$ is closer to the robber than $v$, and since we have assumed $d_1 \leq k$. 
Since the shortest path from $v$ to the robber must include $r$, we have that $0=d_2+d_3- d_1$. 
Note that if the robber was in $T_i$, then $d_2+d_3- d_1$ would be twice the distance from $r$ to the path between the robber and $v$, which would therefore be non-zero. 
Thus, the cops can detect that the robber is not in $T_i$ in this case. 

Now suppose the robber is on $T_i$.
If $d_2=1$, then the robber is captured, either because $d_1=0$ or they are at the unique vertex at distance $d_3-1$ from $v$ and distance $1$ from $r$. 
If $d_2 \geq 2$, then after the robber's next move, then the robber will still be within $T_i$. 
The cops now begin the strategy again on $T_i$, a tree of strictly smaller height. 
By repeating this process, the cops inductively reduce the height of the tree until they capture the robber at a leaf vertex.
\end{proof}

For a set of vertices $S$, let $N(S)$ denote the set of vertices in $V(G)\setminus S$ with a neighbor in $S$. We will refer to this as the \emph{border} of $S$. 
The \emph{vertex-isoperimetric parameter} of a graph $G$ at volume $k$, $\Phi_V(G,k)$, was defined in  \cite{BMMN} as
\[\Phi_V(G,k) = \min_{S \subseteq V(G) : |S| = k} |N(S)|. \] 
In the literature, it is common to define the term \emph{isoperimetric inequality}, which is any bound on $|N(S)|$ from below and usually for a general set $S$. 

The authors in \cite{BMMN} also defined a \emph{vertex-$h$-index} of a graph $G$, $H_v(G)$, which they used to construct a lower bound on $\prox_k(G)$, and thus also $\zeta_k(G)$, when $k=1$.
This result could be generalized for general $k$, but we give an improvement to this result. 
Let $\Phi_V(G) = \max_k \Phi_V(G,k)$.

\begin{theorem} \label{thm:isoperimetric}
For any graph $G$, 
    \[\zeta_k(G) \geq \prox_k(G) > \frac{\Phi_V(G)}{\Delta_k(G)}.\]
\end{theorem}
\begin{proof}
Suppose that we play the $k$-proximity game with $\frac{\Phi_V(G)}{\Delta_k(G)}$ or fewer cops. 
Let $S_1$, $S_2$, and $S_3$ be the sets of vertices that the robber could safely occupy during, respectively, the present round before the cops have probed any vertices, after the cops have probed their vertices, and after the robber has moved. 
Since each cop can detect the robber on at most $\Delta_k(G)$ vertices on each round, cumulatively there are at most $\Phi_V(G)$ vertices of $S_1$ in which the cops can catch the robber. Therefore, we can assume $|S_2| \geq |S_1|-\Phi_V(G)$. 

Let $k$ be such that $\Phi_V(G,k)=\Phi_V(G)$. 
    Since adding $i$ vertices to a set may decrease the border by at most $i$ vertices, we have $\Phi_V(G,k+i)\geq\Phi_V(G)-i$. 
    If $i<\Phi_V(G)$, then this gives $\Phi_V(G,k+i)>0$, and since $\Phi_V(G,k+i)=0$ when $k+i \geq |V(G)|$, it follows that $k+\Phi_V(G) \leq |V(G)|$.
    
Consider the case when $|S_1| \geq k+\Phi_V(G)$ and $|S_2| \leq k+\Phi_V(G)$. 
We have that $|S_2| \geq |S_1|-\Phi_V(G) \geq (k+\Phi_V(G))-\Phi_V(G) = k$. Suppose that $|S_2|=k+p$, for some $p\geq 0$. 
We also have that $|N(S_2)| \geq \Phi_V(G,k+p) \geq \Phi_V(G)-p$, and so $$|S_3| = |S_2|+| N(S_2)| \geq (k+p)+(\Phi_V(G)-p) = k+\Phi_V(G).$$ 
Therefore, if we ever start with $k+\Phi_V(G)$ or more vertices unknown to the cop before the current cop move, then there will be $k+\Phi_V(G)$ or more vertices unknown to the cop before the next move. 
Since we start with $|V(G)| \geq k+\Phi_V(G)$ vertices that are unknown to the cop, the cops will never be able to reduce the size of the unknown vertices below $k$, and so the robber will never be caught.
\end{proof}

\section{Trees} \label{sec:trees}

The original Localization game, in which the cops have unrestricted vision, is completely understood on trees. As was proved first in~\cite{seager1} and later in~\cite{nisse2}, the localization number of trees is at most two, and trees with localization number $1$ are completely characterized. In contrast, it was shown in~\cite{BMMN} that the $1$-visibility localization number of trees is unbounded. This also holds for the $k$-visibility localization number. 

\begin{theorem} \label{thm:trees_unbounded}
If $d$ is a positive integer, then a tree $T$ exists such as $\zeta_k(T) > d$.    
\end{theorem}
\begin{proof}
The vertex-isoperimetric value of a $q$-regular tree $T$ has been studied~\cite{iso-peak-trees-Vrto-2010} and it is known to be at least $\frac{3}{40}(\mathrm{rad}(T)-2)$. Theorem~\ref{thm:isoperimetric} then yields 
$$\zeta_k(T) \geq \prox_k(T)> \frac{3}{40}\frac{(\mathrm{rad}(T)-2)(q-2)}{q(q-1)^k-2}.$$ 
Since $q$ and $\mathrm{rad}(T)$ are independent, we may fix $q$ and increase $\mathrm{rad}(T)$, yielding a tree with $\zeta_k(T)$ arbitrarily large. 
\end{proof}

Note that when $q$ and $k$ are fixed constants, this lower bound is asymptotic to the radius of the tree. 
We show in Theorem~\ref{thm:upper_tree} that there is an upper bound on $\zeta_k$ that is asymptotic to the radius of the tree for $k$ constant. 
As such, the bounds are asymptotically tight, and so the behavior of $\zeta_k$ on trees is largely analogous to $\zeta_1$. 
These bounds are not asymptotically tight when $k$ changes with the radius of the tree, however. 
Noting that $\zeta_k(G)$ weakly decreases as $k$ increases, in Theorem~\ref{thm:tree_growth_characterization} we find the threshold on $k$ in terms of the radius required for the $k$-visibility localization number to diverge on trees, and show that a class of graphs exists that diverge beyond this threshold.  

We begin with an example problem, which demonstrates the case when $k$ is large. 

\begin{theorem}\label{thm:trees_simple}
For a tree $T$, if $k = \rad(T)$, then $\zeta_{k}(T) \leq 2$.
\end{theorem}

\begin{proof}
We give a strategy in which two cops with visibility $\rad(T)$ capture the robber. Choose a vertex $r_1$ that has distance at most $\rad(T)$ to each vertex, and root $T$ at $r_1$. The first cop probes $r_1$, which will always return a distance $d_1 \leq \rad(T)$. The robber is restricted to a subtree of $T$ rooted at some neighbor of $r_1$: if they attempt to move from one neighbor of $r_1$ to another, they must pass through $r_1$, at which point $d_1 = 0$ and the first cop locates the robber. The second cop sequentially probes the neighbors of $r_1$ until they receive a distance $d_2 < d_1$ on some vertex $r_2$. The cops now know that the robber is located within the subtree rooted at $r_2$ and can repeat the process on this subtree. In this way, the cops confine the robber to subtrees of decreasing height until the robber is captured on a leaf vertex. 
\end{proof}

Consider a sequence of trees $(T_n)_{n \ge 1}$ of increasing radius. Theorem~\ref{thm:trees_unbounded} tells us for fixed $k$, it may be that $\lim_{n \to \infty} \zeta_k(T_n) \to \infty$. However, if we allow $k$ to grow as $k = k(n) = \rad(T_n)$, then we see $\lim_{n \to \infty} \zeta_k(T_n) \le 2$. In the following theorem, the main result of this section, we determine how fast $k$ must grow as a function of $\rad(T)$ to ensure $\zeta_k(T_n)$ remains bounded.

\begin{theorem}\label{thm:tree_growth_characterization}
For every sequence of trees $(T_n)_{n \ge 1}$ with $\rad(T_n) \to \infty$ and every continuous function $f(n) = \Omega(\sqrt{n})$, there exists a constant $d$ such that if $k = k(n) = f(\rad(T_n))$, then
    \[ \limsup_{n \to \infty} \zeta_k(T_n) \le d. \]
    
 For every continuous function $f(n) = o(\sqrt{n})$ there exists a sequence of trees $(T_n)_{n \ge 1}$ such that, with $k = k(n) = f(\rad(T_n))$,
    \[ \lim_{n \to \infty} \zeta_k(T_n) \to \infty.\]
\end{theorem}

The upper bound of Theorem~\ref{thm:tree_growth_characterization} follows from the following general upper bound for $\zeta_k(T)$ in terms of $k$ and $\rad(T)$.

\begin{theorem}\label{thm:upper_tree}
    For any tree $T$ and any $k \geq 1$, $\zeta_{k}(T) \leq \lceil\frac{\mathrm{rad}(T) + k}{k^2}\rceil + 1$. 
\end{theorem}

\begin{proof}
We will give a strategy in which $\lceil\frac{\mathrm{rad}(T) + k}{k^2}\rceil + 1$ $k$-visibility localization cops capture the robber.
Choose a vertex $r$ of $T$ with distance at most $\rad(T)$ from any other vertex, and root $T$ at $r$. 

First, consider the following strategy to protect a path from $r$ to a leaf of $T$ of length $\ell$ using $\lceil \frac{\ell + k}{k^2} \rceil$  cops. We claim that by implementing this strategy, after a finite number of rounds, the robber cannot enter the path without being within distance $k$ of some cop, which we refer to as the cops \emph{detecting} the robber. 
Consider the set of vertices $S_k = \{s_i\}_{i=0}^{\lfloor \ell/k\rfloor}$ on the path such that $s_i$ is at distance $ik$ from $r$. 
If $\ell/k$ is not an integer, 
then add the leaf to $S_k$ as $s_{\lceil \ell/k\rceil}$. 
We claim if each of these vertices is probed at least once every $k$ rounds, then after each vertex has been probed once, the cops can detect the robber entering the path. Suppose that the cops begin by probing the vertices along the path at distance $0, k^2, 2k^2, \ldots, \lceil \frac{\rad(T)-k^2+k}{k^2}\rceil k^2$ from $r$. Note that this placement uses
\[ 1 + \left\lceil \frac{\rad(T)-k^2+k}{k^2}\right \rceil = 1 + \left\lceil \frac{\rad(T)+k}{k^2} - 1\right \rceil = \left\lceil \frac{\rad(T)+k}{k^2}\right \rceil \]
probes. 
On subsequent rounds, if a cop probed $s_i$ on the previous round, then they probe $s_{i+1}$ on this round. On the round after a cop probes the leaf, they return to the root and probe $r$. 
This strategy guarantees that every vertex in $S_k$ is probed every $k$ moves. 
If $\rad(T)/k$ is an integer, then $S_k$ contains $1+\rad(T)/k = 1+\lceil \rad(T)/k\rceil$ vertices, and otherwise it contains $2+ \lfloor \rad(T)/k \rfloor = 1 + \lceil \rad(T)/k \rceil$ vertices.
Over $k$ rounds, the cops probe a total of 
\[ k \cdot \left\lceil \frac{\rad(T)+k}{k^2}\right \rceil 
= k \cdot \left\lceil \frac{\rad(T)/k+1}{k}\right \rceil \ge \lceil \rad(T)/k+1 \rceil 
= 1 + \lceil \rad(T)/k \rceil\]
vertices, where we use the inequality $x \cdot \lceil \frac{y}{x} \rceil \ge \lceil y \rceil$.
    
Now suppose the cops have implemented this strategy for at least $k$ rounds so that every vertex in $S_k$ has been probed at least once. Assume for contradiction that the robber enters the path at some vertex $v$ between the vertices $s_i$ and $s_{i+1}$ and then leaves the path, all without any cop probing a vertex within distance $k$ of the robber. Let round $t$ be the last round a cop probes $s_{i+1}$ before the robber reaches $v$. 
Let the robber's position on round $t$ be $u_1$, and the robber's position on round $t+k$ be $u_2$. To go undetected on round $t$, $u_1$ must be at distance at least $k+1$ from $s_{i+1}$. To avoid being detected on round $t-1$, when $s_i$ was probed, $u_1$ must also be at distance at least $k$ from $s_i$. First suppose $u_1$ is at distance at most $\frac{k}{2}$ from $v$. There is exactly one path from $u_1$ to $s_i$ and exactly one path from $u_1$ to $s_{i+1}$, both of which must pass through $v$. Thus, we can say that $d(s_i,v) = d(s_i,u_1) - d(v,u_1) \geq k - \frac{k}{2} = \frac{k}{2}$ and similarly, $d(s_{i+1},v) = d(s_{i+1},u_1) - d(v,u_1) \geq k + 1 - \frac{k}{2} = \frac{k}{2} + 1$. This gives
\[d(s_i,s_{i+1}) = d(s_i,v) + d(v,s_{i+1}) \ge \frac{k}{2} + \frac{k}{2} + 1 = k+1,\]
but $s_i$ was chosen to be at distance at most $k$ from $s_{i+1}$. Thus, we conclude $u_1$ is at distance greater than $k/2$ from $v$. We have that $d(u_1,v)+d(u_2,v) \le k$ as the robber has only moved $k$ times from round $t$ to round $t+k$. 
Repeating the same analysis with $u_2$, we once again imply $d(s_i,s_{i+1}) \ge k+1$, which is a contradiction. 
We conclude that the robber cannot enter the path at any vertex $v$. The robber also cannot enter the path at any $s_i \in S_k$ as if they are at distance $k+1$ when $s_i$ is probed, they are unable to reach $s_i$ before it is probed again. We conclude that after $k$ rounds if the robber enters the path, some cop will probe a vertex within distance $k$ of the robber.

Next, we extend the strategy provided above that detects the robber if it moves to a path of $T$ to a strategy that detects the robber on $T$. 
Order the root-to-leaf paths in $T$ using a depth-first-search ordering of the leaves. 
All but one cop, say $C^\ast$, will probe these paths. 
The cops use the first $k$ rounds to clear the first path using the method described above. 
During this time, $C^\ast$ probes $r$ every round. 
After the first $k$ rounds, the cops probing the first path will now focus on the second path. 
They spend the next $k$ rounds clearing the second path using the method described above while $C^\ast$ alternates between probing $r$ and the vertex furthest from $r$ shared between the first and second paths. 
The cops continue this strategy, spending $k$ rounds per path while $C^\ast$ alternates between probing $r$ and the last vertex shared between the current and previous path.

We claim that this detection strategy works; that is, after finitely many rounds, some cop will probe a vertex within distance $k$ of the robber. To prove this claim, we show that from round $jk+1$ to round $(j+1)k$; the robber can only stay undetected by occupying vertices on paths with index larger than $j$. 
As $T$ contains finitely many leaves, eventually, the robber will be detected. 
Because the vertices were chosen by depth-first search, the vertices of path $j$ partition $T$ into the vertices on paths less than $j$ and those on paths larger than $j$, so the robber cannot access the vertices on paths with index smaller than $j$ without first occupying a vertex on path $j$. The initial portion of path $j$, up to at most distance $k$ from where path $j$ diverges from path $j+1$, has been probed for at least the last $k$ rounds, so the robber is unable to enter undetected by the previous analysis. 
As $C^\ast$ probes the junction between path $j$ and path $j+1$ every other round, and $k \ge 1$, the robber cannot access vertices close to the junction either. 
The remaining vertices of path $j$ are only adjacent to vertices in paths with index smaller than $j$, so the claim holds by induction.  

To capture the robber, the cops will implement a detection strategy on the subtree rooted at each neighbor of $r$ with the modification that $C^\ast$ probes $r$ whenever they would probe the root of the subtree. The cops' goal is to determine a neighbor $r'$ of $r$ such that the robber is contained in the subtree of the descendants of $r'$, with root taken to be $r'$. 
Suppose the detection strategy is played on the subtree of descendants of $r'$ and some cop's probe returned a distance of at most $k$ to the robber. 
There is a cop that probes $r'$ once in every $k$ rounds, and the cop $C_\ast$ probes $r$ once every two rounds, and might probe a vertex close to $r'$ on the other round. All other vertices that are probed by the cops have distance more than $k$ to vertices outside of this subtree. 
Therefore, the cop who probed a distance of at most $k$ to the robber knows the robber is on this subtree unless it was one of these specially mentioned cops. In the latter case, the cops probe $r$ and $r'$ in the next round. If the robber was not on this subtree in the last round, then the cop on $r$ will probe a distance of at most $k$, say $d_1$, and the cop on $r'$ will probe a higher value, say $d_2$, which has either $d_2 = d_1+1$ or $d_2= \ast$, both of which indicate to the cops that the robber is not on this subtree. 
If the robber was on this subtree, then either both of these cops probe $\ast$, or 
$d_2<d_1$, or $d_2=k$ and $d_1=\ast$. 

The cops can then probe $r$ and $r'$, ensuring the robber is not on either vertex, and then restart the strategy on the subtree rooted at $r'$. 
As this subtree has a strictly smaller radius, the claim follows by induction, noting that a tree of radius one is a star with $k$-visibility localization number $1$ when $k \ge 1$. 

Suppose the robber plays only on the subtree rooted at some $r'$. 
In that case, the cops will eventually use the detection strategy on that subtree, detect the robber, and make progress by limiting the robber to a smaller subtree. 
To avoid this, the robber will attempt to move from one neighbor of $r$ to another. As $C^\ast$ probes $r$ every other round and $k \ge 1$, they will detect the robber at distance $1$ from $r$. When this occurs, the cops shift to a straightforward strategy: $C^\ast$ probes $r$ every round while the other cops probe neighbors of $r$. As long as $C^\ast$ received a distance of $1$ on the previous round if a cop probes a neighbor of $r$ and receives either no distance or a distance larger than $C^\ast$, then the cops know that the robber is not located on that subtree. 
We conclude that the robber approaching $r$ only speeds up the cops' progress to find the subtree on which the robber resides. Therefore, the robber's best strategy is to stay on one subtree of $r$ where they will eventually be detected, and the cops will make progress.
\end{proof}

The following theorem gives a bound in the case $k=0.$

\begin{theorem}\label{thm:upper_tree_0}
    For any tree $T$, $\zeta_{0}(T) \leq \mathrm{rad}(T) + 1$.
\end{theorem}
\begin{proof}
    We may modify the proof of Theorem~\ref{thm:upper_tree}, such that we sequentially search through each root-to-leaf path, placing a cop on each vertex of the path. 
\end{proof}

The lower bound in Theorem~\ref{thm:tree_growth_characterization} follows from the following construction. Let $T(h,q)$ denote the complete $q$-ary tree of height $h$; that is, the rooted tree where each vertex of distance less than $h$ from the root has $q$ children. 

\begin{theorem}\label{thm:lower_tree} 
For each $h\geq 3$, $q\geq 4$, and $k$, there exists a tree $T$ of height $\mathrm{rad}(T)=h(2k+1) - (k+1)$ with 
\[\zeta_k(T)>\frac{\Phi_V(T(h,q))}{4(2k+1)} \geq \frac{3}{160} \frac{h-2}{2k+1}.\]
As such, there is a sequence of trees $(T_j)_{j \geq 1}$ with $\mathrm{rad}(T_j) \rightarrow \infty$ and sequence of integers $k_j = \sqrt{\mathrm{rad}(T_j)}/\omega_j$ such that
\[
\lim_{j \rightarrow \infty}\zeta_{k_j}(T_j) = \omega(1),
\]
where $\omega_j \rightarrow \infty$ as $j \rightarrow \infty$. 
\end{theorem}
\begin{proof}
Let $T$ be the result of taking $T(h,q)$ and subdividing each edge $2k+1$ times, except if the edge contains a leaf vertex, in which case it is only subdivided $k$ times. Note that the resulting tree has height $\mathrm{rad}(T)=h(2k+1)-(k+1)$. The vertices with degree at least $q$ will be called \emph{major} vertices. 

We play the $k$-proximity game on $T$, referred to as the \emph{main game}, while at the same time observing a separate game on $T$ that is played alongside the main game, referred to as the \emph{shadow game}. We assume the robber plays the same moves in both games. Suppose the cops have a winning strategy with $c$ cops when playing the regular game on $T$. This strategy will be converted to a cop strategy of the shadow game on $T$ with $4c(2k+1)$ cops. We will show that if the robber cannot avoid capture in the regular game, then the robber is also captured in the shadow game. 
We will then show that playing the shadow game on $T$ is equivalent to playing the $0$-proximity game on $T(h,q)$. The $0$-proximity game on $T(h,q)$ requires more than $\Phi_V(T(h,q))$ cops for the cops to be able to win by Theorem~\ref{thm:isoperimetric}. The isoperimetric peak has been studied for $q$-ary regular trees of height $h$ in \cite{iso-peak-trees-Vrto-2010}, which showed that $\Phi_V(T') \geq \frac{3}{40}(h-2)$ when $q \geq 4$ and $h \geq 3$. This result will give us the desired bound. 

The rules of the shadow game are as follows. Each cop is a $k$-proximity cop; however, the cops may only play on rounds $\alpha(2k+1)+1$, for $\alpha \geq 0$. During all other rounds, the cops skip their move, and so no cops are played. In addition, the cops may only play on the major vertices of $T$. The rules for the moves and capture of the robber are identical to the regular $k$-proximity game. 

Given a winning cop strategy of the main game on $T$, we will now construct a winning strategy in the shadow game on $T$. 
We partition the rounds of the main game into periods of $2k+1$ rounds. 
For period $\alpha\geq 0$, which runs from round $\alpha(2k+1)+1$ to round $(\alpha+1)(2k+1)$, let $S_{\alpha}$ be the set of major vertices such that some cop was placed on a vertex of distance at most $2k$ from that major vertex during period $\alpha$. 
Note that there are at most two major vertices of distance $2k$ from each vertex, there are $(2k+1)$ rounds per period, and during each round, $c$ cops are played. 
As such, $S_\alpha$ contains at most $2c(2k+1)$ vertices. 
To construct the cop strategy in the shadow game, play cops on the major vertices $S_{\alpha-1} \cup S_\alpha$ on round $\alpha(2k+1)+1$, which uses at most $4c(2k+1)$ cops, which is the number of cops we have for the shadow game. 

We now show that this cop strategy for the shadow game on $T$ will indeed capture the robber. 
Suppose that under some optimal robber strategy in the regular game, the robber would be captured on vertex $v$ by a cop probing vertex $u$ at time $\alpha(2k+1)+\beta$, with $1 \leq \beta \leq 2k+1$, but that the robber (using the same moves as in the regular game) would not be captured in the shadow game at time $(\alpha +1)(2k+1) +1$ or earlier. 
Let $u_1$ and $u_2$ be the two major vertices of distance at most $k$ from $u$ (set $u_1 = u_2 = u$ if $u$ is a major vertex). 
Without loss of generality, we may assume that $v$ has distance at most $k$ from $u_1$. 
As $u_1$ and $u_2$ are probed on round $\alpha(2k+1)+1$ and on round $(\alpha+1)(2k+1)+1$ in the shadow game, the robber will be captured in the shadow game if the robber has distance $k$ or less on either of these two rounds. This means that on round $\alpha(2k+1)+\beta'$, the robber is at a distance at most $k+1-\beta'$ from $u_1$ or $u_2$ if $1 \leq \beta' \leq k+1$, and at a distance at most $\beta'-k-1$ from $u_1$ or $u_2$ if $k+2 \leq \beta' \leq 2k+1$. 
This also implies that the robber could not have landed on $u_1$ or $u_2$ during period $\alpha$ without later being captured in the shadow game on round $(\alpha+1)(2k+1)$. 
As the robber would be captured if it was between $u_1$ and $u_2$ on round $\alpha(2k+1)+1$ and the robber cannot reach $u_1$ or $u_2$ in period $\alpha$ without being later captured, we can assume the robber is not on any vertex in the path of length $2k+1$ from $u_1$ to $u_2$ at any time during period $\alpha$ or was captured. 

The above implies the robber was distance at most $k$ from some other major vertex $u_3 \notin \{u_1,u_2\}$ in round $\alpha(2k+1)+1$, and moved onto vertices of distance at most $k$ from $u_1$ by round $\alpha(k+1) + \beta$. 
Note that this implies $u_3$ has distance $2k+1$ from $u_1$. 
The robber was then captured in the main game. 
We may assume the robber could not have elected to stay on the unique vertex $w$  of distance $k$ from $u_3$ and distance $k+1$ from $u_1$ for the rounds after round $\alpha(2k+1)+1$ until its capture, as we assume the robber avoids capture whenever possible. 
Hence, it must have been that another cop played at distance $k$ from $w$ during period $\alpha$ in the normal game to force the robber to move from $w$.  

If this were the case, then this cop must have played within distance $2k$ from $u_3$ during period $\alpha$ in the main game. 
This implies that $u_3 \in S_\alpha$, but then $u_3$ is played in the shadow game in round $\alpha(2k+1)+1$. 
However, the robber has distance at most $k$ from $u_3$ (this defines $u_3$), and so the robber must have been captured in the shadow game in round $\alpha(2k+1)+1$

To finish the proof, we show that the shadow game on $T$ is equivalent to the $0$-proximity game on $T(h,q)$. 
To see this, suppose in the $0$-proximity game that the robber may be on a vertex in $V_{\alpha-1}'$ on round $\alpha$ immediately before the cops' move, and also suppose in the shadow game that the robber may be on a vertex in $N_k(V_{\alpha-1}')$ on round $\alpha(2k+1)+1$ immediately before the cops' move. 
The cops' play on the set of vertices $U_\alpha$ in both games at this point. 
Note that this satisfies the conditions of the shadow game since these are major vertices in the shadow game and are played on round $\alpha(2k+1) +1$, and also note that there are no valid moves in the shadow game that do not obey these constraints. 
This means that the robber can only be on a vertex in $V_\alpha = V_{\alpha-1}' \setminus U_\alpha$ in the $0$-visibility game if it was not captured. 
Similarly, the robber can only be on a vertex in $N_k(V_{\alpha-1}') \setminus N_k(U_\alpha) = N_k(V_\alpha)$ in the shadow game if it was not captured. 
Now, the robber moves. 

In the $0$-visibility game, the robber moves once and is now on a vertex in $V_\alpha' = N(V_\alpha)$. 
In the shadow game, the robber moves $2k+1$ times, and is now on a vertex in $N_{2k+1}(N_k(V_\alpha)) = N_k(N_{2k+1}(V_\alpha))$. 
Note that in $T$, the vertices of distance $2k+1$ from the major vertices $V_\alpha$ are either the major vertices in $V_\alpha$ or the major vertices $V_\alpha' \setminus V_\alpha$. 
As such, $N_k(N_{2k+1}(V_\alpha)) = N_k(V_\alpha')$. 
This means that in either game, the same moves may be played on both games, and the possible robber locations immediately before the cops play in the shadow game are exactly determinable from the possible robber locations immediately before the cops play in the $0$-visibility game, and vice versa. 
Note that the robber is caught in the $0$-visibility game only if $V_\alpha = \emptyset$ for some $\alpha$, which means that the robber would be on $N_k[\emptyset] = \emptyset$ in the shadow game, meaning the robber would also be caught in the shadow game. 
The reverse also holds, in that if the robber is captured in the shadow game, then $N_k(V_\alpha') = \emptyset$, which implies $V_\alpha' = \emptyset$, and the robber is captured in the $0$-visibility game. 
This completes the proof. 
\end{proof}

We note that a complete $q$-ary tree of height $h'=h(2k_1) - (k+1)$ will contain the subdivision of $T(h,q)$ that we analyzed in Theorem~\ref{thm:lower_tree}, and so Lemma~\ref{lem:subtreebound} yields the following result. 
\begin{corollary}
For $h'\geq 7k+3$ and $q\geq 4$, 
\[\zeta(T(h',q)) > \frac{3}{160} \frac{h'-k}{(2k+1)^2}.\]
\end{corollary}

This combines with
Theorem~\ref{thm:upper_tree} to yield the following result. 
\begin{corollary}
For $q \geq 4$ and any $k\geq 3$, any complete $q$-ary tree $T$ of height $h\geq4$ has
\[
\zeta_k(T) = \Theta\left(\frac{\mathrm{rad}(T)}{k^2}\right).
\]
\end{corollary}

We complete this section with a proof of Theorem~\ref{thm:tree_growth_characterization}.

\begin{proof}[Proof of Theorem~\ref{thm:tree_growth_characterization}]
We begin with the case that $f = \Omega(\sqrt{n})$. There is $\alpha \in \mathbb{R}^+$ and $N \in \mathbb{N}$ such that for all $n \ge N$, $f(n) \ge \alpha\sqrt{n}$. As $f$ is continuous, there are $\beta, \gamma > 0$ such that $\beta \le f(x) \le \gamma$ for all $x \in [1,N]$. By Theorem~\ref{thm:upper_tree}, if $\rad(T_n) < N$,
\[ \zeta_k(T_n) \le \left\lceil\frac{\mathrm{rad}(T_n) + k}{k^2}\right\rceil + 1 \le \frac{N+\gamma}{\beta^2}+2. \]
For $\rad(T_n) \ge N$, we have that
\[ \zeta_k(T_n) \le \left\lceil\frac{\mathrm{rad}(T_n) + k}{k^2}\right\rceil + 1 \le \frac{\rad(T_n)}{f(\rad(T_n))^2}+\frac{1}{f(\rad(T_n))} + 2 \le \frac{1}{\alpha^2}+\frac{1}{\alpha\sqrt{N}} + 2. \]
Thus, by setting
\[ d = \max\left(\frac{N+\gamma}{\beta^2}+2, \frac{1}{\alpha^2}+\frac{1}{\alpha\sqrt{N}} + 2\right) \]
we have that
\[ \limsup_{n \to \infty} \zeta_k(T_n) \le d. \]

If $f = o(\sqrt{n})$, then for each integer $z$, let $(T_{z,n})_{n \ge 1}$ be the sequence of trees guaranteed by Theorem~\ref{thm:lower_tree}. With $k = k(n) = \rad(T_{n,n})$, we have that
\[ \lim_{n \to \infty} \zeta_k(T_{n,n}) \to \infty. \]
The proof follows. \end{proof}

\section{Subdivisions of graphs} \label{sec:subdiv}

The relationship between subdividing a graph and the Localization game has been studied~\cite{car}, and it is known that for any graph $G$, the graph obtained by subdividing each edge of $G$ $3n$ times, $G^{1/3n}$, has $\zeta(G^{1/3n}) = 1$. 
The central idea of the cop's strategy on $G^{1/3n}$ is that if the cop probes all the original vertices of $G$ one at a time, the robber will eventually be identified as being close to a probed vertex, and cannot move to another original vertex of $G$ without the cop identifying its location. 
The issue with extending this strategy to the $k$-visibility Localization game for a fixed $k$ is that typically we will consider graphs where $3n$ is much larger than $k$. 
This renders the strategy from the original game useless since the cop can no longer see the entirety of any of the paths of length $3n$. 
Therefore, we need to find a new strategy for subdividing graphs in the limited visibility game.  

To remedy this, we will instead consider what happens if we subdivide the edges non-uniformly. 
That is to say, we do not necessarily subdivide all the edges the same number of times. 
The advantage to this approach is that if we separate the vertices of a graph into two sets and subdivide all the edges that go from one set to the other a large number of times, then in the proximity game, the cops can first clear these subdivided edges and then focus on clearing one of the sets of vertices. 
It will take a large number of rounds for the contaminated area to spread back over the subdivided edges so the cops can essentially ignore half of the graph while the contamination spreads back over the subdivided edges. Our main results are summarized in the following two theorems.

\begin{theorem} \label{thm:proxsub}
For a tree $T$ and positive integer $k,$ there is a subdivision $T'$ of $T$ with $\mathrm{prox}_k(T') = 1$.
\end{theorem}

\begin{theorem} \label{thm:zetatreesubdivision}
For a tree $T$ and positive integer $k > 1$, there is a subdivision $T'$ of $T$ with $\zeta_k(T') = 1$.
\end{theorem}

We begin with the proof of Theorem~\ref{thm:proxsub}. In the following, let $\text{Children}(u)$ denote the set of children of $u$ in some rooted tree, $\text{Parent}(u)$ denote the parent vertex of $u$, and $\text{Desc}(u)$ be the set of descendants of $u$ (including $u$ itself). 

\begin{proof}[Proof of Theorem~\ref{thm:proxsub}]
    Assume that $k < \rad(T)$, or else the problem is trivial. 
    Choose a vertex $r$ of $T$ with distance at most $\rad(T)$ from any other vertex as the root of $T$. 
    Each vertex will be labeled with a word on the alphabet of the non-negative integers. 
    The root is labeled $\varepsilon$. 
    We iteratively label the vertices of $T$ by labeling the children of a vertex $u$ as $ui$, for $1 \leq i \leq |\text{Children}(u)|$. See Figure~\ref{fig:Vertex_labelled_tree} for an example of this labeling.

        \begin{figure}[t]
    \centering
    \begin{tikzpicture}[scale=0.5]
    \GraphInit[vstyle=Classic]
    \SetUpVertex[FillColor=white]

    \tikzset{VertexStyle/.append style={minimum size=8pt, inner sep=1pt}}
    
    \Vertex[x=-1,y=1,L={$\varepsilon$}]{V0}
    
    %---------------------------------

    \Vertex[x=-6,y=-2,L={$3$}]{V02}
    
    \Vertex[x=0,y=-2,L={$2$}]{V01}

    \Vertex[x=4,y=-2,L={$1$}]{V00}
   
    %---------------------------------

    \Vertex[x=-8,y=-4,L={$33$}]{V022}
    
    \Vertex[x=-6,y=-4,L={$32$}]{V021}

    \Vertex[x=-4,y=-4,L={$31$}]{V020}
    
    \Vertex[x=-1,y=-4,L={$22$}]{V011}
    
    \Vertex[x=1,y=-4,L={$21$}]{V010}
    
    \Vertex[x=4,y=-4,L={$11$}]{V000}

    %---------------------------------

    \Vertex[x=-6,y=-6,L={$321$}]{V0210} 
    
    \Vertex[x=4,y=-6,L={$111$}]{V0000}

    %---------------------------------

    \Vertex[x=-8,y=-8,L={$3212$}]{V02101} 
    \Vertex[x=-4,y=-8,L={$3211$}]{V02100} 
    
    \Vertex[x=4,y=-8,L={$1111$}]{V00000}

    %= = = = = = = = = = = = = = = = =

    \Edge(V0)(V00)
    \Edge(V0)(V01)
    \Edge(V0)(V02)
    
    \Edge(V00)(V000)
    \Edge(V000)(V0000)
    \Edge(V0000)(V00000)
    
    \Edge(V01)(V011)
    \Edge(V01)(V010)
    
    \Edge(V02)(V022)
    \Edge(V02)(V021)
    \Edge(V02)(V020)
    
    \Edge(V021)(V0210)
    
    \Edge(V0210)(V02101)
    \Edge(V0210)(V02100)

\end{tikzpicture}

    \caption{A tree labeled according to the proof of Theorem~\ref{thm:proxsub}.}
    \label{fig:Vertex_labelled_tree}
\end{figure}
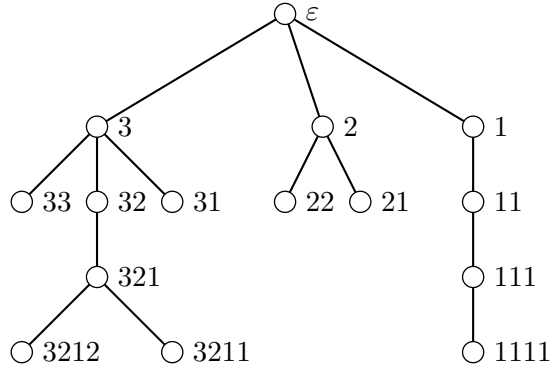

For each edge $(\text{Parent}(v),v)$ in $T$, we need to determine $x_v$, the number of times we will subdivide this edge when constructing $T'$. 
To calculate each of the values of $x_{v}$ required, we will provide a strategy in $T'$. 
As this strategy plays out, it will give lower bounds on each of the $x_{v}$, which will be sufficient for the strategy to be successful. 
 Let $P_{v}=(v_0=\text{Parent}(v),v_1,v_2, \ldots, v_{x_{v}+1}=v)$ denote the path of length $x_{v}+1$ in $T'$ between $\text{Parent}(v)$ and $v$, 
and also denote the subpath of $P_{v}$ with the first $k+1$ vertices removed as $P'_{v}=(v_{k+1},v_{k+2}, \ldots, v_{x_{v}+1})$. 

To define each of these $x_v$, it is necessary to know the amount of time taken by the given strategy to clear the descendant subtree of any given vertex. 
For any vertex $u$ in $T$, let $t_u$ be the number of rounds needed to ensure the robber is 
not on the subtree $\text{Desc}(u)$ while assuming the following: 
\begin{enumerate}
\item before the first round, $P_u$ does not contain the robber; 
\item the robber does not move from a vertex outside of $\text{Desc}(u)\cup P'_{u}$ onto a vertex in $\text{Desc}(u)\cup P'_{u}$ during these $t_u$ rounds; and 
\item the cop must ensure that if the robber moves from a vertex in $\text{Desc}(u)$ to a vertex outside of $\text{Desc}(u)$, then the robber will be captured by the end of the $t_u$ rounds.
\end{enumerate} 

The strategy on $T'$ is defined recursively as follows. 
Let $u$ be some vertex of $T$ such that the robber is known to be on some vertex in $\text{Desc}(u)$.
For each $i$, starting with $i=|\text{Children}(u)|$ and decreasing $i$ by one each time until $i=1$, the cop plays to ensure that the edges along the path $P_{ui} = (u_0=u,u_1,u_2, \ldots, u_{x_{ui}+1}=ui)$ are clear while ensuring that the robber cannot move from an uncleared path onto $u$ without being captured. 
To do this, in round 1 of this procedure, the cop probes $u=u_0$. This ensures the robber is distance at least $k+1$ from $u$.
The robber now moves to a vertex of distance at least $k$ from $u$. 
On round $t$ for $2 \leq t \leq 2k+1$, the cop probes the vertex $u_{(2t-2)k}$ of path $P_{ui}$ if $t\cdot k<x_{ui}+1$, and probes $ui$ otherwise. 
In the former case, the robber can only be on a vertex $u_j$ if $j \geq (2t-1)k+1$. 
The robber moves and may move to a vertex between $u_{(2t-1)k}$ and $ui$ at the end of round $t$. 
Since during round $t+1$, the cop probes $u_{2tk}$, which has distance at most $k$ to the vertices $u_{(2t-1)k}, \ldots, u_{(2t+1)k}$. 
This is unless the cop probes vertex $ui$, but this only occurs when $ui$ is within distance $k$ from all vertices $u_{(2t-1)k}, \ldots, ui$. 
If the robber was on an uncleared path extending from $u$, then it would take the robber until round $k+1$ to move to $u$. 
Thus, by the end of round $2k+1$, the robber may have moved through $u$ to any vertex of distance $k$ from $u$, but no further. 
Also, the robber may be on a vertex between $u_{4k^2+1}$ and $ui$, inclusive. 

We now repeat the procedure, starting with the cop probing vertex $u$ in round $(2k+1)+1$. In round $(2k+1)+t$ with $2 \leq t \leq 2k+1$, the cop probes the vertex $u_{4k^2 + (2t-1)k}$ of path $P_{ui}$ if $4k^2 + (2t-1)k<x_{ui}+1$, and probes $ui$ otherwise. 
We continue repeating this procedure, with the cop probing either vertex $u_{(4k^2+k)\alpha-k}$ or $ui$ at the end of the $\alpha$ repetition of the procedure, at which point $(2k+1)\alpha$ rounds have taken place.
The first repetition where $ui$ will be probed is when $\alpha = \lceil \frac{x_{ui}+1+k}{4k^2+k}\rceil$, and so we consider $P_{ui}$ cleared after these $y_{ui}=(2k+1)\lceil\frac{x_{ui}+1+k}{4k^2+k}\rceil \leq \frac{x_{ui}}{2k-1}$ rounds, although we note that those vertices of distance at most $k$ from $u$ may contain the robber at this point. These vertices of distance at most $k$ from $u$ will be cleared during the next cop move, unless they are already clear.

        \begin{figure}[H]
    \centering
    \begin{tikzpicture}[scale=0.5]
    \GraphInit[vstyle=Classic]
    \SetUpVertex[FillColor=white]

    \tikzset{VertexStyle/.append style={minimum size=8pt, inner sep=1pt}}
    
    \Vertex[x=-1,y=1,NoLabel=true]{V0}
    
    %---------------------------------

    \Vertex[x=-6,y=-2,NoLabel=true]{V02}
    
    \Vertex[x=0,y=-2,NoLabel=true]{V01}

    \Vertex[x=4,y=-2,NoLabel=true]{V00}
   
    %---------------------------------

    \Vertex[x=-8,y=-4,NoLabel=true]{V022}
    
    \Vertex[x=-6,y=-4,NoLabel=true]{V021}

    \Vertex[x=-4,y=-4,NoLabel=true]{V020}
    
    \Vertex[x=-1,y=-4,NoLabel=true]{V011}
    
    \Vertex[x=1,y=-4,NoLabel=true]{V010}
    
    \Vertex[x=4,y=-4,NoLabel=true]{V000}

    %---------------------------------

    \Vertex[x=-6,y=-6,NoLabel=true]{V0210} 
    
    \Vertex[x=4,y=-6,NoLabel=true]{V0000}

    %---------------------------------

    \Vertex[x=-8,y=-8,NoLabel=true]{V02101} 
    \Vertex[x=-4,y=-8,NoLabel=true]{V02100} 

    %= = = = = = = = = = = = = = = = =

    \Edge[label={$6$}](V0)(V01)
    
    \Edge[label={$6$}](V01)(V011)
    \Edge[label={$0$}](V01)(V010)
    
    \Edge[label={$0$}](V0)(V00)
    
    \Edge[label={$0$}](V00)(V000)
    
    \Edge[label={$0$}](V000)(V0000)    
    
    \Edge[label={$12$}](V0)(V02)
    
    \Edge[label={$12$}](V02)(V022)
    \Edge[label={$6$}](V02)(V021)
    \Edge[label={$0$}](V02)(V020)
    
    \Edge[label={$0$}](V021)(V0210)
    
    \Edge[label={$6$}](V0210)(V02101)
    \Edge[label={$0$}](V0210)(V02100)

\end{tikzpicture}

    \caption{The tree from Figure~\ref{fig:Vertex_labelled_tree} with labels representing the number of subdivisions for a $k=2$ strategy.}
    \label{fig:subdivided_tree}
\end{figure}
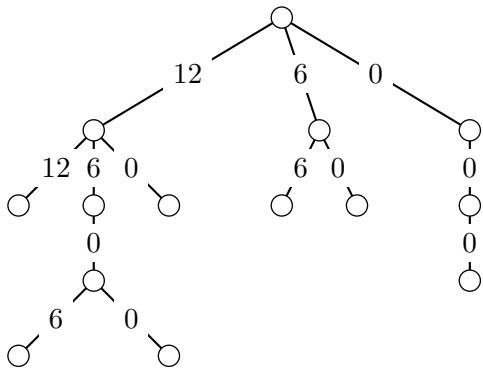

This is repeated for each $i$, from $i = |\text{Children}(u)|$ until $i=1$, at which point all the paths $P_{ui}$ have been cleared, noting that this cop strategy does allow the robber to enter a path $P_{ui}$ through the vertex $ui$ after this path had been cleared. 
Define $x_{u0}=1$, and 
\[
x_{ui} = \left(\sum_{i'<i}y_{ui'}\right) + \max_{i'<i}\{t_{ui'}\},  
\]
noting that this is well defined since $x_{ui'}$ is defined for $i'<i$, since $y_{ui'}$ is defined once $x_{ui'}$ is defined, and since the $t_{ui'}$ are defined through the induction. 
Notice that after we have cleared path $P_{ui}$ in the above procedure, there will be a further $\sum_{i'<i}y_{ui'}$ rounds until all the paths $P_{ui}$ have been cleared, and so the robber may have moved along $P_{ui}$ from $ui$ to be at the vertex of distance $x_{ui} - \sum_{i'<i}y_{ui'} = \max_{i'<i}t_{ui'}$ from $u$, but no closer. 
Thus, after all of these paths have been cleared, in the next $t_{uj}$ rounds, we may clear the subtree $\text{Desc}(uj)$, where $j$ is the smallest index such that the robber may still be in $\text{Desc}(uj)$. 
At this point, the robber may have moved close to the root but has only just had enough time to reach the root. 
We then repeat this procedure from the start, again clearing the paths $P_{ui}$ for all $i$, and then clearing the next subtree $\text{Desc}(u(j+1))$, until all subtrees have been cleared. 
Since the robber only has enough time to reach $u$ before we restart, the robber cannot move between two subtrees, ensuring that the tree will be cleared by the end of this procedure. 
The robber will be captured using this procedure if it started in $\text{Desc}(u)$. 
Define $t_u$ to be the length of time it took to run this procedure. 

Performing this iterative approach, starting with $u$ to be the root vertex, each edge $(v,vi)$ of the graph has its parameter $b_{vi}$ explicitly defined. 
The tree $T'$ with edge $(v,vi)$ subdivided $b_{vi}$ times requires exactly one cop to win, and so the proof is complete. 
\end{proof}

As an immediate implication of Lemma~\ref{lem:proxplusone} and Theorem~\ref{thm:proxsub}, it is possible to subdivide any tree to achieve $\zeta_k(T) \leq 2$, but when is it possible for $\zeta_k(T)$ to be $1$? If we consider the structure of a tree in which every edge has been subdivided a large number of times, then we can see that locally the tree will resemble a spider. This will prove advantageous for the cop since anytime $k \geq 2$, spiders only require a single cop, as we will see in Lemma~\ref{lem:spider}.

\begin{lemma} \label{lem:spider}
If $S$ is a spider and $k \geq 2$, then $\zeta_k(S) = 1$.
\end{lemma}

\begin{proof}
To find the robber, the cop uses the following strategy. First, they select a branch arbitrarily and probe its leaf, then the vertex at distance $k$ from the leaf, then distance $2k$ from the leaf, and continue until the unique vertex with degree greater than $2$ is within distance $k$ of the cop. If the robber was visible at any point before the cops' last probe, then their location is determined since the cop had exactly one possible robber location at each visible distance. If the cop sees the robber on their probe that is within distance $k$ and the robber's position is not immediately determined, then they know the robber is on the other side of the vertex of degree greater than $2$, and so they can continue as if they had probed the unique vertex of degree at least $3$, which will be covered in a later case. 

Now the cop selects a new branch and alternates between probing the high degree vertex and probing the leaf, then the vertex at distance $k-1$ from the leaf, and so on until that branch is cleared or the cop sees the robber while on the high degree vertex. If the cop sees the robber with distance $\ell$ while on the high-degree vertex, then on their next move, they select an arbitrary uncleared branch and probe the vertex at distance $\ell+1$ from the high-degree vertex. If the robber was on that branch, then they are immediately located. If there is no vertex at distance $\ell+1$ from the high-degree vertex on that branch, then the cop instead probes the leaf on this branch and achieves the same outcome.\end{proof}

Now consider what happens if before we subdivide to reduce the $k$-proximity number to $1$, we subdivide the tree so that all vertices of degree greater than $2$ are very far apart. The tree is now locally a spider, and if the edges of the tree have been subdivided enough times, the robber will not be able to escape to another high-degree vertex before the cop locates them. This is exactly the strategy described below in the proof of Theorem~\ref{thm:zetatreesubdivision}. 

\begin{proof}[Proof of Theorem~\ref{thm:zetatreesubdivision}]
Take every edge of $T$ and subdivide it $\frac{Nk}{2} + k + 1$ times, where $N$ is the smallest value such that $H_N$ (the $N^{th}$ harmonic number) is at least $\frac{4\Delta(T)}{k}$. Next, apply the subdivision described in Theorem~\ref{thm:proxsub} to obtain $T'$. The tree $T'$ now satisfies $\mathrm{prox}_k(T') = 1$ and the vertices of degree greater than $2$ in $T'$ are all sufficiently far apart for our purposes. 
When the cop wins the proximity game and sees the robber, we assume the cop is on a high degree vertex (degree greater than $2$) or the cop is not on a high degree vertex but the distance from the cop to the robber is greater than the distance from the cop to the nearest high degree vertex. 
We do this since if the cop sees the robber far from a high-degree vertex then the strategy plays out similarly. Now the cop uses the fact that this is locally a spider graph and applies the second part of the strategy from Lemma~\ref{lem:spider}, but rather than using the end of the branch the cop uses the furthest vertex the robber could be on. If the robber cannot evade for more than $\frac{Nk}{2} + 1$ rounds after being seen, then the cop can use this strategy to eventually capture the robber.

We now show that the robber cannot evade for more than $\frac{Nk}{2} + 1$ rounds once the cop has won the proximity game. First, consider how long it will take to clear the first branch. Since the robber could have initially been on only one vertex on the branch, there are now three possible locations, and since $k \geq 2$ the cop can clear this in one move. For the first $\frac{k}{4}$ branches, the cop will be able to do this, and if we account for the fact that every other round the cop will probe the high-degree vertex, we can see that this will take at most two rounds per branch or $\frac{k}{2}$ rounds total. At this point, the contamination stretches across a length of $k+1$ on each branch, and it will now take two rounds to clear each branch, which becomes at most four rounds once we include the rounds where the cop probes the high degree vertex. We can see that $\frac{k}{8}$ branches can be cleared like this before it takes more moves, then the contamination will stretch across a length of $2k+1$ on each branch while the cop clears the next $\frac{k}{12}$ branches in six rounds each, and so on. Continuing this pattern, we can see that to clear all the branches in time, it must be that $\frac{k}{4} \sum_{i=1}^N \frac{1}{i} \geq \Delta(T)$ which implies that we need $H_N \geq \frac{4\Delta(T)}{k}$. Thus, the furthest the robber could have moved is $\frac{Nk}{2} + 1$, which gives the desired result.
    \end{proof}

Since any graph can be subdivided enough to make a graph that is locally a spider, we can apply the concept from Theorem~\ref{thm:zetatreesubdivision} any time there is a class of graphs which is closed under subdivisions and where it is possible to subdivide enough to reduce $\mathrm{prox}_k$ to $1$. This is formalized in Corollary~\ref{cor:subdivclass}.

\begin{corollary} \label{cor:subdivclass}
For any family of graphs $\mathcal{F}$ that is closed under subdivisions, if for all $G \in \mathcal{F}$ there is an integer $k > 1$ such that there exists a subdivision $G'$ of $G$ with $\mathrm{prox}_k(G') = 1$, then there exists a graph $G'' \in \mathcal{F}$ which is also a subdivision of $G$ such that $\zeta_k(G'') = 1$. 
\end{corollary}

As an example, the class of cycle graphs $\{C_n\}_{n \geq 3}$ is closed under subdivisions, and every cycle can be subdivided in a way that reduces $\mathrm{prox}_2$ to $1$. Therefore, by Corollary~\ref{cor:subdivclass} it is possible to subdivide any cycle in a way that reduces $\zeta_2$ to $1$.

A question that arises from the statement of Theorem~\ref{thm:zetatreesubdivision} is why $k$ is now greater than $1$, unlike in Theorem~\ref{thm:proxsub} where $k$ could be any positive integer. It can be shown that the class of graphs $G$ with $\zeta_1(G) = 1$ is exactly the set of caterpillars~\cite{mm}. As a result, it is not always possible to subdivide a tree to reduce $\zeta_1$ to $1$.

\section{Cartesian Grids}

In this section, we provide a lower and an upper bound on the $k$-visibility localization number for the $n \times n$ Cartesian grid graphs, which we label $G_{n, n}$. These bounds are close to being tight, and for $n$ sufficiently large, we find the exact value of the $k$-visibility localization number for these graphs in most cases, and we find it is one of two values in the other cases, as in the following theorem. 

\begin{theorem} \label{thm:main_grids_results}
Let $n$ and $k$ be positive integers with $n\geq 2k^2+2k+1$ when $k\geq 2$, and $n\geq 2(2k^2+2k+1)=10$ when $k=1$. 
If $n\pmod{2k^2+2k+1} \in [1,2k^2]$, then \[\zeta_k(G_{n, n}) = \Big\lfloor\frac{n}{2k^2+2k+1}\Big\rfloor.\]
Otherwise, 
\[\zeta_k(G_{n, n}) \in \Big\{\Big\lfloor\frac{n}{2k^2+2k+1}\Big\rfloor,\Big\lfloor\frac{n}{2k^2+2k+1}\Big\rfloor+1\Big\}.\] 
\end{theorem}

We do note that the proof of Theorem~\ref{thm:main_grids_results} is lengthy. As a high-level overview, we will split the game into two phases. 

Phase 1 revolves around seeing the robber for the first time. This is represented by playing the $k$-proximity game on the graph. 
In this phase, the locations that the robber could be in without having been seen will be analyzed. 
Let $S$ be this set of vertices at a certain time. 
The cop strategy that we will provide will ensure that the border of $S$ will contain (at most) $n$ vertices, with (at most) one border vertex per column. 
The strategy will also ensure that each cop removes $2k^2+2k+1$ vertices from the set $S$ (by 'seeing' these vertices), with exceptions for when a cop plays near the edges of the grid. 
This will be done in a way that guarantees that more than $n$ vertices of $S$ are `seen' by the cops on the average round, and that when we remove the `seen' vertices, with a resulting set of `unseen' vertices $S'$, the border of $S'$ will contain (at most) $n$ vertices, with (at most) one border vertex per column.
Playing recursively, the cops will eventually see the robber for the first time.

In Phase 2, the cops are able to use the fact that some cop has just seen the robber to capture the robber under the typical $k$-visibility Localization game rules.

For the case that $k=1$, Theorem~\ref{thm:main_grids_results} improves the upper bound of the $1$-proximity number and the $1$-visibility localization number on $n \times n$ Cartesian grid graphs found in previous work \cite{BMMN}.

\subsection{A Tiling of the Infinite Grids by $k$-Balls} \label{ssec:TilingInfiniteGrids}

\begin{figure}[h]
    \centering
    \includegraphics[scale=0.15]{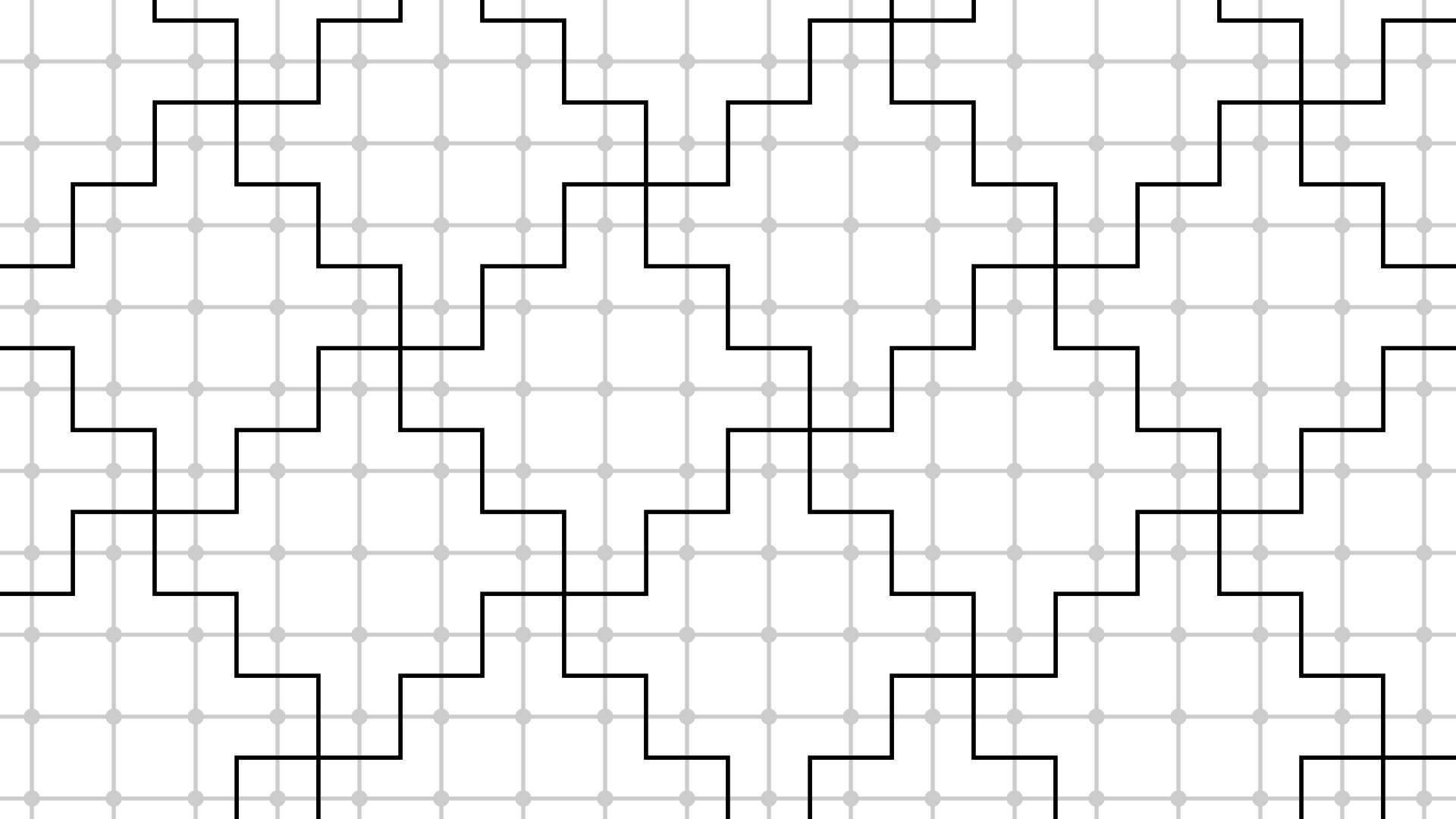}
    \caption{An example of the $2$-tiling of the grid.}
    \label{fig:tiling_of_grid}
\end{figure}

The \emph{square grid graph} $G_\infty$ is the graph on vertex set $\mathbb{Z} \times \mathbb{Z}$ with edges between two vertices $(x,y)$ and $(x',y')$ when $|x-x'|+|y-y'|=1$. 
The \emph{rectangular grid graph} of length $n$, $G_{\infty,n}$, is the induced subgraph of the square grid graph on the vertex set $\mathbb{Z}\times[n]$. The $n \times n$ Cartesian grid graph $G_{n,n}$ that we will be focusing on is the induced subgraph of the square grid graph on the vertex set $[n]\times[n]$. 
Grids are thought of as Cartesian products of paths. Cartesian products are typically displayed with vertex $(x,y)$ being placed $x$ units above and $y$ units to the right of the origin $(0,0)$. 
This is in contrast to the Cartesian coordinate system.

A \emph{vertex-partition} of a graph $G$ is a  set of subsets of $V(G)$, $\{V_1, V_2, \ldots\}$,  such that $\bigcup_{i\geq 1} V_i = V(G)$ and $V_i \cap V_j=\emptyset$ for $i\neq j$. 
Since the graphs we are working on have planar embeddings, the vertex partitioning can be used to create a \emph{tiling} of the plane. We abuse notation by calling the $V_i$ \emph{tiles}, and refer to the set $\{V_i : i \geq 1\}$ as a \emph{tiling} of $G$. 

We initially provide a tiling of the square grid graph and then convert this over to tilings of the rectangular grid graphs. 
Let $G$ be any of these grid graphs or the $n \times n$ Cartesian grid graph. 
Define the $k$-tile about $(x,y) \in \mathbb{Z} \times \mathbb{Z}$ as $T_{x,y} = T^k_{x,y} = \{(x',y') \in V(G): |x-x'|+|y-y'|\leq k\}$. 
Note that on the rectangular grid graphs and the $n \times n$ Cartesian grid graphs, $(x,y)$ does not have to be a vertex of the grid: The $3$-tile $T^3_{0,0}$ is the subset of vertices $\{(1,1),(2,1),(1,2)\}$ of $G_{n\times n}$ when $n \geq 2$. 
A vertex-decomposition into $k$-tiles of one of our grids will be called a \emph{tiling} of that grid. 
We note that such a tiling on $G_\infty$ is equivalent to a perfect $k$-error-correcting Lee code on $\mathbb{Z}^2$, and the $k$-tiles are sometimes called Lee spheres~\cite{GolombWelch}.

\begin{definition}
Let \[S = \{(ik+j(k+1),i(k+1)-jk): i,j \in \mathbb{Z}\}\] and let 
\[S_{i',j'} = \{(ik+j(k+1), i(k+1)-jk): i\geq i' \text{ and } j \geq j' \text{, or } j>j' \text{ and } i\in \mathbb{Z}\}.\]
\end{definition}
See Figure \ref{fig:half_tiling_of_grid} for a visual representation of the $2$-tiling that results from $S_{0,0}$ with $k=2$ by placing $2$-tiles centered at each vertex in $S_{0,0}$.  

\begin{figure}[h]
    \centering
    \includegraphics[scale=0.15]{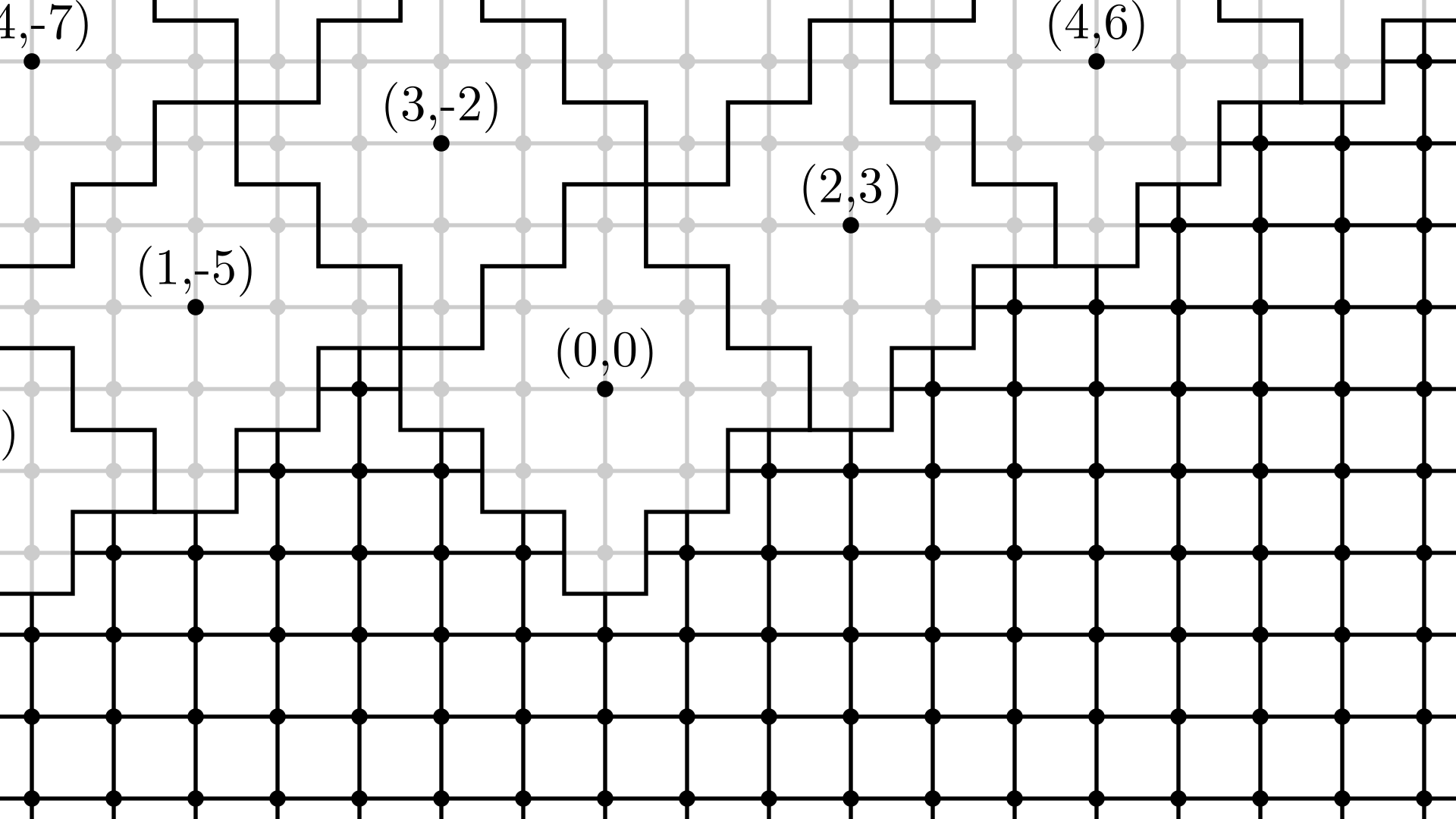}
    \caption{A half $2$-tiling of $G_\infty$ around $(0,0)$, defined using the $2$-tiles around each point in $S_{0,0}$.}
    \label{fig:half_tiling_of_grid}
\end{figure}

When playing the $k$-proximity game throughout the rest of the paper, we can think of the vertices that may contain an unseen robber as an infection, which spreads to the neighboring vertices after each cop move. 
With this idea, when a cop is placed on some vertex, the infection is cleaned for all vertices of distance at most $k$ from that cop. 

Then in this subsection and the next, we will show the following. 
\begin{enumerate}
    \item If an infection was on the vertices of distance at most $k$ from some vertex in $S_{i',j'}$ and the infection spread to neighboring vertices, then the newly infected vertices are exactly below a vertex that was infected; 
    \item The resulting infected vertices can also be decomposed into a set of disjoint tiles, which can be represented as a downward shift of the tiles around the points in $S_{i'',j''}$ for certain $i'',j''$; and
    \item If the infection in certain balls of radius $k$ are removed from the set of infected vertices, then the remaining infection can also be decomposed into a set of disjoint tiles, which can also be represented as a downward shift of the tiles around the points in $S_{i'',j''}$ for certain $i'',j''$.
\end{enumerate}

For vertex $u$, define $N[u]$ to be the vertices of distance at most $1$ from $u$, and $N[U] = \bigcup_{u \in U} N[u]$. The \emph{border} of the set $U$ is 
$N(U) = N[U] \setminus U$. 
Define $D^i[U]= \{(x-i,y) : (x,y) \in U\}$ for any set of vertices $U$ in a square or rectangular grid graphs, where $i \in \mathbb{Z}$. Let $D[U]=D^1[U]$.  
The \emph{downward-border} of the set $U$ is 
$D(U) = D[U] \setminus U$.

For each $(x,y)$, we decompose the (potential) vertices of distance $k+1$ from $(x,y)$ as four sets. 
Let $f_{ne}(x,y) = \{(x+i,y+k+1-i): 0 \leq i \leq k\}$, 
$f_{nw}(x,y) = \{(x+k+1-i,y-i): 0 \leq i \leq k\}$, 
$f_{sw}(x,y) = \{(x-i,y-k-1+i): 0 \leq i \leq k\}$, and 
$f_{se}(x,y) = \{(x-k-1+i,y+i): 0 \leq i \leq k\}$. 
See Figure~\ref{fig:BorderTxy} for an example of these sets in relation to $T_{x,y}$ when $k=2$. 
These four sets are pairwise disjoint, and it is straightforward to find that $N(T_{x,y}) = f_{ne}(x,y) \cup f_{se}(x,y) \cup f_{sw}(x,y) \cup f_{nw}(x,y)$, and that each vertex in these sets has distance $k+1$ from $(x,y)$ in $G_\infty$.  

We also note that the vertices in $f_{ne}(x,y)$ each have distance $k$ to the vertex $(x+k,y+k+1)\in S$, and similarly the vertices of $f_{se}(x,y)$ each have distance $k$ to the vertex $(x-k-1,y+k)\in S$, 
the vertices of $f_{sw}(x,y)$ each have distance $k$ to the vertex $(x-k,y-k-1)\in S$, 
and the vertices of $f_{nw}(x,y)$ each have distance $k$ to the vertex $(x+k+1,y-k)\in S$.

\begin{figure}[h]
    \centering
    \includegraphics[scale=0.2]{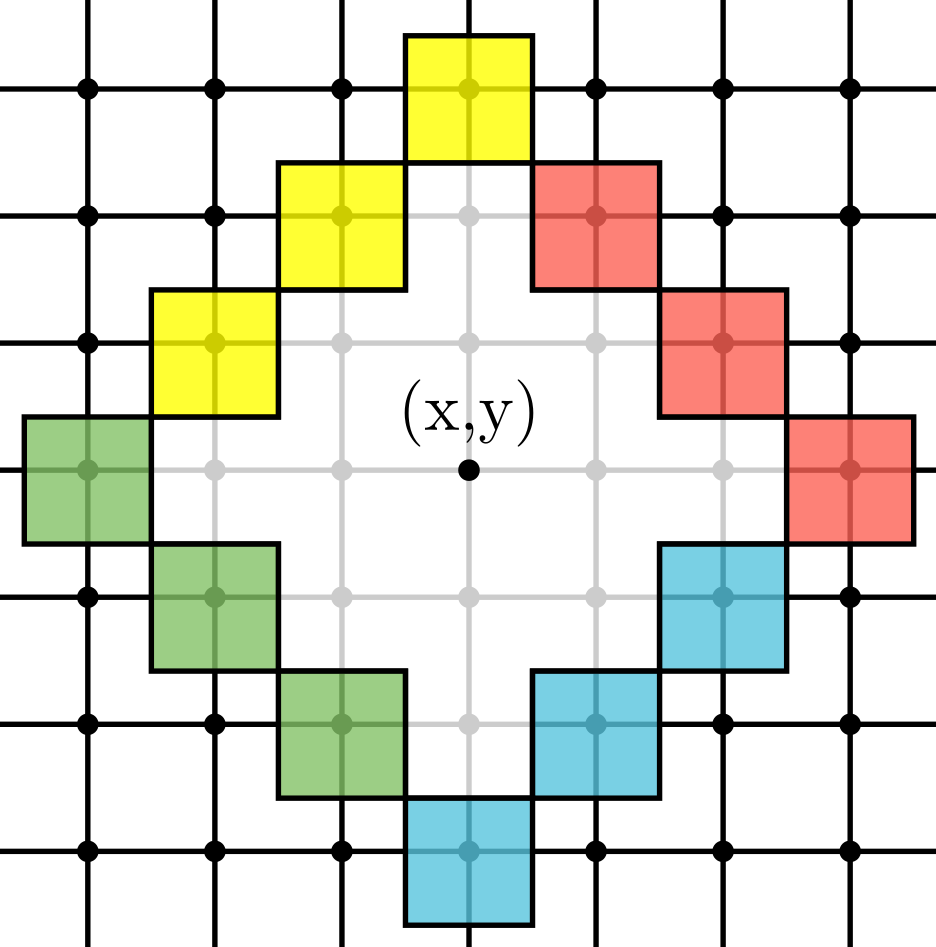}
    \caption{A $2$-tile $T_{x,y}$ along with the four subsets of vertices $f_{ne}(x,y)$ (in red), $f_{se}(x,y)$ (in blue), $f_{sw}(x,y)$ (in green), and $f_{nw}(x,y)$ (in yellow).}
    \label{fig:BorderTxy}
\end{figure}

\begin{theorem}[\cite{GolombWelch}] \label{thm:kTiling}
The set of $k$-tiles $\{T_{x,y} : (x,y) \in S\}$ form a tiling of $G_\infty$. 
\end{theorem}

It will be useful to prove the following technical lemmas. 

\begin{lemma} \label{lem:U_D_is_N}
For some $(i',j') \in S$, let $U_{i',j'}=\bigcup_{(x,y)\in S_{i',j'}} T_{x,y}$. We then have that \[D(U_{i',j'}) = N(U_{i',j'}).\]
\end{lemma}
\begin{proof}
    It follows from Theorem~\ref{thm:kTiling} that each vertex in $N(U_{i',j'})$ has distance $k+1$ from some vertex $(x,y)\in S_{i',j'}$ and has distance $k$ to some vertex $(x',y') \in S \setminus S_{i',j'}$. 

    For each vertex $(x,y)\in S_{i'+1,j'} \subseteq S_{i',j'}$, the three vertices $\{(x+k,y+k+1), (x-k,y-k-1), (x+k+1,y-k)\}$ are each in $S_{i',j'}$, and so $N(T_{x,y}) \cap N(U_{i',j'}) \subseteq f_{se}(x,y)$. Noting that $f_{se}(x,y) \subseteq D(T_{x,y})$, we have $N(T_{x,y}) \cap N(U_{i',j'}) \subseteq D(T_{x,y})$. 
    Thus, if there is a vertex in $N(U_{i',j'})$ but not in $D(U_{i',j'})$, then it must have distance $k+1$ from some vertex in $S_{i',j'}$ but cannot have distance $k+1$ from a vertex in $S_{i'+1,j'}$.

    There is a single vertex in $S_{i',j'} \setminus S_{i'+1,j'}$, namely $(x,y)=(i'k+j'(k+1),i'(k+1)-j'k)$. 
    The two vertices $\{(x+k,y+k+1), (x+k+1,y-k)\}$ are each in $S_{i',j'}$ and the vertices $\{(x-k-1,y+k), (x-k, y-k-1)\}$ are not, and so $N(T_{x,y}) \cap N(U_{i',j'}) \subseteq f_{se}(x,y) \cup f_{sw}(x,y)$. We have that $(f_{se}(x,y) \cup f_{sw}(x,y)) \setminus (x,y-k-1) \subseteq D(T_{x,y})$ and $(x,y-k-1)\subseteq D(T_{x-2k-1,y+1})$. 

    We have thus shown that each vertex in $N(U_{i',j'})$ is also in $D(T_{x,y})$ for some $(x,y)\in S_{i',j'}$, and so $N(U_{i',j'}) \subseteq \bigcup_{(x,y) \in S_{i',j'}} D(T_{x,y})$. 
    Since the tiles $T_{x,y}$ over all $(x,y)\in S_{i',j'}$ are a tiling of $U_{i',j'}$, $D(U_{i',j'})= \left(\bigcup_{(x,y)\in S_{i',j'}} D(T_{x,y}) \right) \setminus U_{i',j'}$, and so $N(U_{i',j'}) \subseteq D(U_{i',j'})$, as we needed to show.         
\end{proof}

\subsection{Tilings of the Rectangular Grid Graphs} \label{ssec:TilingRectangularGrids}

\begin{figure}
    \centering
    \includegraphics[scale=0.15]{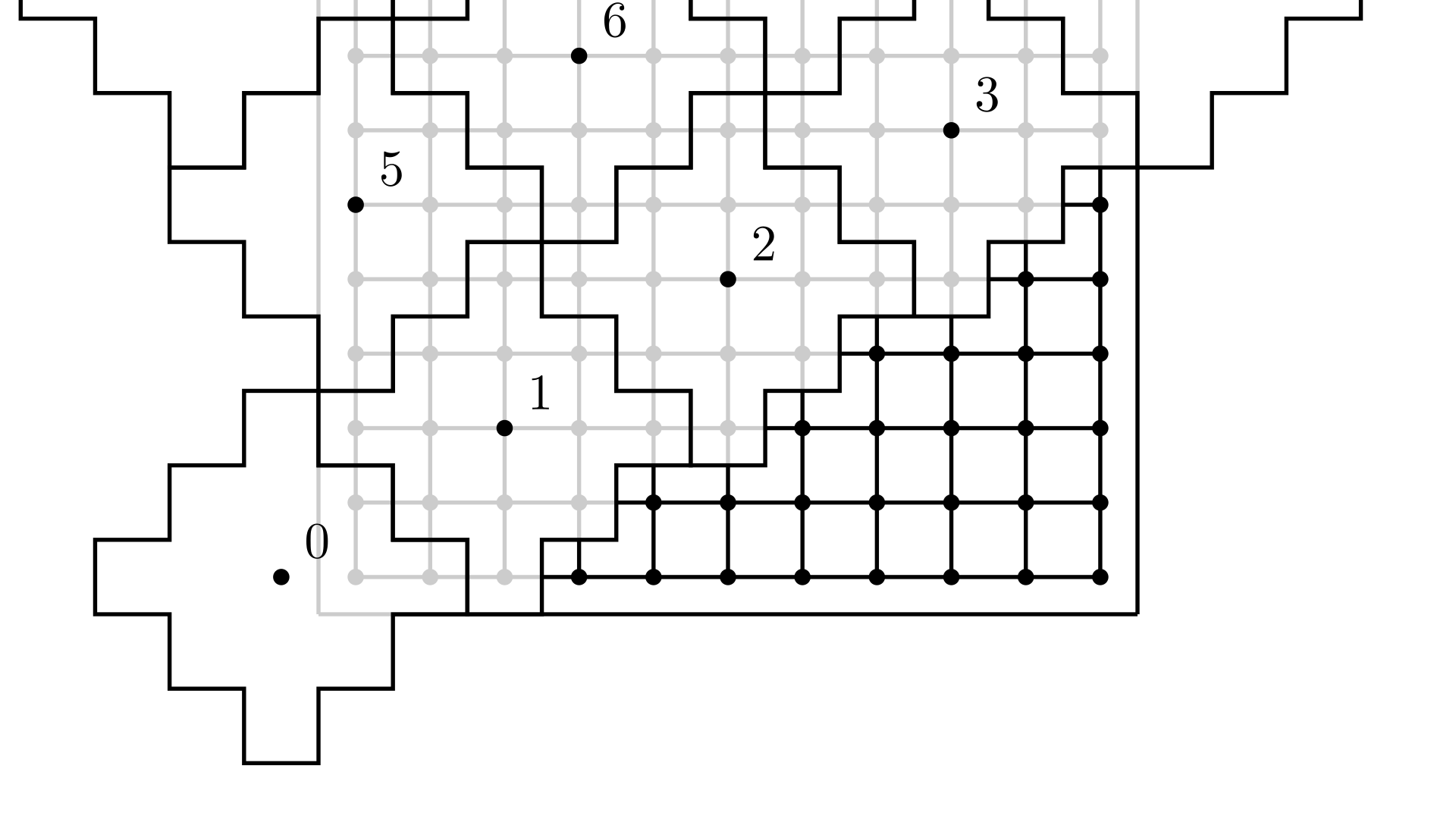}
    \caption{An example of the $2$-tiles $T_i$ for $i \geq 0$ in $G_{\infty,11}$.}
    \label{fig:rectangle_labels}
\end{figure}

We will use the tiling of the square grid graph to find a tiling of the rectangular grid graphs; see Figure \ref{fig:rectangle_labels}. 
The corresponding tiles that form a tiling for $G_{\infty,n}$ are obtained by taking the set of tiles $\{T_{x,y} \cap V(G_{\infty,n}) : (x,y)\in S\}$. 
Note in particular that $T_{x,y} \cap V(G_{\infty,n}) \neq \emptyset$ if and only if $y \in [1-k,n+k]$, and so we may define the set $S' = \{(x,y)\in S: y\in [1-k,n+k]\}$ so that the set of tiles $\{T_{x,y} \cap V(G_{\infty,n}) : (x,y)\in S'\}$ is the aforementioned tiling of $G_{\infty,n}$. 
We assign a single integer label $i$ to each of these tiles in the tiling of $G_{\infty,n}$, so that $T_i=T_{x,y}$ for a particular $(x,y) \in S$,  with the following scheme:
\begin{enumerate}
    \item $T_0=T_{0,0}\cap V(G_{\infty,n})$; 
    \item If $(x,y),(x+k,y+k+1) \in S'$, then $T_i=T_{x,y}$ if and only if $T_{i+1}=T_{x+k,y+k+1}$; 
    \item Suppose that $(x,y)\in S'$ and $(x+k+1,y+k)\notin S'$, and that $T_i = T_{x,y}\cap V(G_{\infty,n})$. With $\alpha=\lfloor \frac{x-1}{k+1} \rfloor$, note that $(x+k+1-\alpha k,y-k-\alpha(k+1))\in S'$. 
    Define $T_{i+1} = T_{x+k+1-\alpha k, y-k-\alpha(k+1)}\cap V(G_{\infty,n})$.
\end{enumerate}
We note that $T_i$ is defined for each $i$ an integer. 
Define $S'_j = \{(x,y) \in S' : T_{x,y}=T_i \text{ and } i\geq j\}$. Figure \ref{fig:rectangle_labels} is a representation of the tiles around $S'_0$.  
Note that $S'_{j+1} \subseteq S'_{j}$.

Define $U_p = \bigcup_{i \geq p} T_i$. For the remainder of this subsection, we will give some technical lemmas that show how the neighbors of a set interact with the function $D$, which can be thought of as a function that ``pushes'' a set of vertices down. This will also include how these two functions interact with the infinite set $U_p$. To begin, we may show that the order of these two operations does not matter. 

\begin{lemma} \label{lem:Swap_U_N}
For any subset of vertices $U \subseteq V(G_{\infty,n})$, it holds that $N[D[U]] = D[N[U]]$. 
\end{lemma}
\begin{proof}
If $(v_1,v_2) \in N[D[U]],$ then there is a $(w_1,w_2) \in D[U]$ with $(v_1,v_2) \in N((w_1,w_2))$. 
Since $(v_1,v_2) \in N((w_1,w_2))$, defining $\varepsilon_1=v_1-w_1$ and $\varepsilon_2 = v_2-w_2$, it follows that $|\varepsilon_1|+|\varepsilon_2| =1$.

Since $(w_1,w_2) \in D[U]$, we have $(w_1+1,w_2)\in U$. It follows that $(w_1+\varepsilon_1+1,w_2+\varepsilon_2) \in N[U]$. 
However, we have that $(w_1+\varepsilon_1,w_2+\varepsilon_2)=(v_1,v_2) \in D[N[U]]$. 
Therefore, we have that $N[D[U]] \subseteq D[N[U]]$. 

The opposite direction is done similarly. 
\end{proof}

Lemma~\ref{lem:U_D_is_N} also applies to the rectangular grid graphs, and so by repeated applications of Lemma~\ref{lem:Swap_U_N}, the following result holds. 
\begin{corollary} $N[D^q[U_p]] = D[D^q[U_p]] = D^{q+1}[U_p]$. \label{cor:N_is_D}
\end{corollary}

If we take two sets and `push' both of them down by $q$ steps, then basic set operations will be respected by the push-down operator $D$, with set-minus being one of those.  
\begin{lemma} \label{lem:down_respects_setminus}
For any subsets of vertices $U,V \in V(G_{\infty,n})$, it holds that $D^q[U] \setminus D^q[V] = D^q[U \setminus V]$. 
\end{lemma}
\begin{proof}
If $(u_1,u_2) \in D^q[U] \setminus D^q[V]$, then 
$(u_1+q,u_2) \in U$ and  $(u_1+q,u_2) \notin V$. It follows that $(u_1+q,u_2) \in U \setminus V$, and so $(u_1,u_2) \in D^q[U \setminus V]$, yielding $D^q[U] \setminus D^q[V] \subseteq D^q[U \setminus V]$. The other direction is done similarly. 
\end{proof}

We then have the following from Lemma~\ref{lem:down_respects_setminus}.

\begin{corollary} $D^q[U_p] \setminus D^q[T_p] = D^q[U_{p+1}]$. \label{cor:Setminus-tile}
\end{corollary}

\begin{corollary} \label{cor:subsets_have_larger_index}
$D^q[U_{p_1}] \subseteq D^q[U_{p_2}]$ when $p_1 > p_2$.
\end{corollary}

The final result of this section reveals the symmetric nature of the tilings of the rectangular grid graphs. One way of thinking about this is that for any given tile $T$, there will be a tile $T'$ that is also in the tiling that can be obtained by applying exactly $q+2k^2+2k+1$ downward operations to the original tile $T$. There will be $n+2k$ tiles (including $T$ but not $T'$) between $T$ and $T'$ in the ordering of the tiles. This is encapsulated in the situation where we have already pushed the current set of tiles, $U_p$, down by $q$.

We need the following lemma.
\begin{lemma} $D^{q+2k^2+2k+1}[U_{p+n+2k}]= D^q[U_{p}]$. \label{lem:moving_down_is_moving_along}
\end{lemma}
\begin{proof}
It is sufficient to show that $D^{2k^2+2k+1}[T_{p'+n+2k}] = T_{p'}$ for any $p' \geq p$, since $D^{q+2k^2+2k+1}[U_{p+n+2k}]$ may be decomposed into $\bigcup_{p' \geq p} D^{2k^2+2k+1}[T_{p'+n+2k}]$ and $D^q[U_{p}]$ can be decomposed as $\bigcup_{p' \geq p} D^q[T_{p'}]$. 

For simplicity, we consider just the case where $T_{p'} = T_{x_{1},1-k}$, although the other cases are done similarly. 
With respect to $T_{p'} = T_{x_1,1-k}$, define $$S'_i = \{ (x_1 + (2k+1)i + \alpha k, 1-k + i +\alpha(k+1))  : \alpha \in \mathbb{Z}\} \subseteq S'.$$ 
Note that each $(x,y) \in S'_i$ has $y\equiv (1-k)+i \pmod{k+1}$. 
For each $y\in [1-k,n+k]$, there is a unique $x$ with $(x,y)\in S'_i$ when $i$ satisfies $y\equiv (1-k)+i \pmod{k+1}$, and $(x,y) \notin S'_i$ otherwise.  
This implies that for each $y\in [1-k,n+k]$, there is a unique $x$ with $(x,y)\in \bigcup_{i \in [0,k]} S'_i$. 
As such, $\bigcup_{i\in [0,k]} S'_i$ contains exactly $(n+k) - (1-k) +1 = n+2k$ vertices. 
    
Let $T_{p''}=T_{x_2,1-k}$ be the tile in the tiling of $U_p$ with $x_2>x_1$ such that $x_2$ is as small as possible. 
Note that this means $D^e[T_{p''}] = T_{p'}$ for some $e>0$. 
Since $(x_2, 1-k) \in S'_i$ implies $i \equiv {k+1}$, we must have that $(x_2, 1-k) \in S'_{k+1}$. 
In particular, from the definition of $S'_{k+1}$ we have that $$(x_2,1-k) = (x_1 + (2k+1)(k+1)+\alpha k, 1-k+(k+1) + \alpha(k+1)).$$  
Since we must have $\alpha=-1$, we find $x_2 = x_1 + 2k^2+2k+1$, and so $D^{2k^2+2k+1}[T_{p''}] = T_{p'}$. 
    
We have that $\{T_{x,y} : (x,y) \in \bigcup_{i\in [0,k]} S'_i \} = \{T_j : j \in[p', p''-1]\}$, and since the first set has exactly $n+2k$ elements, so does the second set. We then have that $p'' = p' + n +2k$, and so we have found that $D^{2k^2+2k+1}[T_{p' + n +2k}]=T_{p'}$, as required. 
\end{proof}

\subsection{A Solution for the Finite Grids $G_{n, n}$} \label{ssec:FiniteGrids}

So far, we have explored the tilings for the square and rectangular grid graphs, which are both infinite graphs. In this subsection, we will apply the results on those grids to the finite $n \times n$ Cartesian grid graphs to find an upper bound on $\prox_k(G_{n, n})$.  
In particular, we will create a cop strategy for the $k$-proximity game from the tiling of a rectangular grid, which yields an upper bound on $\text{prox}_k(G_{n,n})$. 

For $G_{n, n}$, define the tiles $T_{x,y}'= T_{x,y} \cap V(G_{n, n})$ and $T_{i}'= T_i \cap V(G_{n, n})$.  
A cop playing the $k$-proximity game on $G_{n, n}$ \emph{removes the tile} $T_{x,y}'$ by playing on the unique vertex $(x',y')\in V(G_{n, n})$ that minimizes $|x'-x|+|y'-y|$. 
Note that each vertex $(x'',y'') \in T_{x,y}'$ satisfies $|x-x''|+|y-y''|\leq k$ by definition, and also has a minimum-length path in $G_\infty$ from itself to $(x,y)$ through the point $(x',y')$. 
Therefore, $|x'-x''|+|y'-y''|\leq k$, and so when the cop plays $(x',y')\in V(G_{n, n})$, the cop either catches the robber or knows that the robber is not on a vertex of $T_{x,y}'$ within $G_{n, n}$ (and hence, the terminology that we ``remove'' the tile, as we ensure the robber is not on these vertices). If we consider the locations where the robber could be as an infection, then removing the tile is equivalent to clearing the infection on those vertices. 

The results of Subsection~\ref{ssec:TilingRectangularGrids} in fact enable us to understand the dynamics of controlling an infection on the rectangular grid $G_{\infty,n}$. 
In Figure~\ref{fig:rectangle_}, we show a portion of the rectangular grid and represent infected vertices as those vertices covered by some tile. This figure displays the removal of two tiles, $T_0$ and $T_1$, followed by the robber moving (or the infection spreading to all neighboring vertices), and then the removal of the next two tiles, $T_2$ and $T_3$. 
This provides a cop strategy, and this section will deal with converting this idea into a strategy to completely clear an infection on $G_{n,n}$. 

The following two corollaries result from inspection using Corollaries \ref{cor:Setminus-tile} and \ref{cor:N_is_D}, respectively. The high-level statement of these corollaries is that removing the infection from tiles on $G_{n, n}$ is equivalent to removing the infection from the respective tile from $G_{\infty,n}$ and that the infection growth of $G_{n, n}$ is the same as the infection growth on $G_{\infty,n}$, except where the resulting vertices are restricted to $V(G_{n, n})$. 

\begin{corollary} \label{cor:cop_move_finite_grid}
    \[
    \left( D^q[U_p] \cap V(G_{n, n}) \right) \setminus \left( D^q[T_p] \cap V(G_{n, n}) \right)
    =
    D^q[U_{p+1}] \cap V(G_{n, n}).
    \]
\end{corollary}

\begin{corollary} \label{cor:robber_move_finite_grid}
    \[
    N\left[ D^q[U_p] \cap V(G_{n, n}) \right] = 
    N\left[ D^q[T_p] \right] \cap V(G_{n, n}) 
    =
    D^{q+1}[U_{p}] \cap V(G_{n, n}).
    \]
\end{corollary}

\begin{figure}
    \centering
    \includegraphics[scale=0.2]{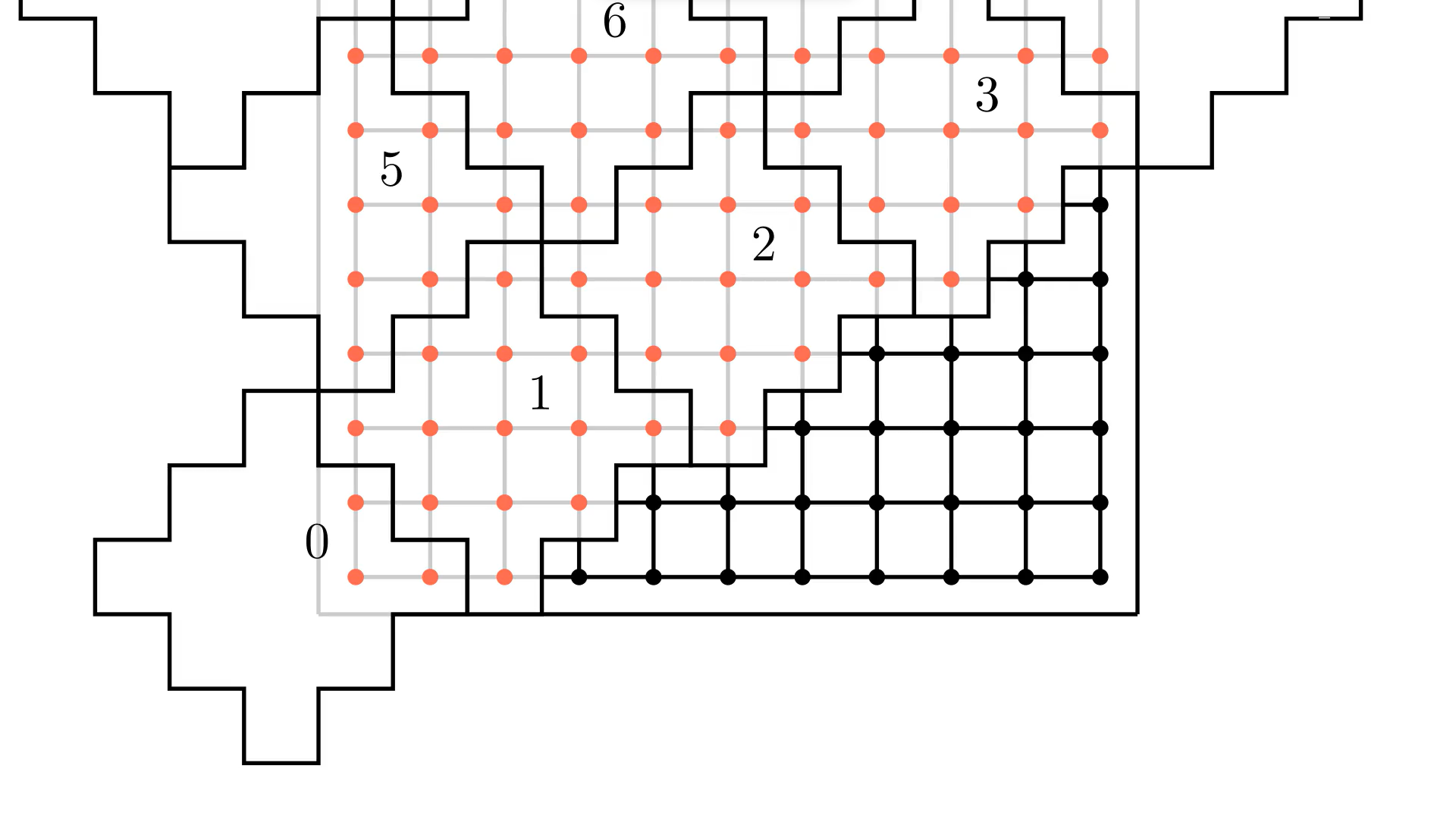}
    \includegraphics[scale=0.2]{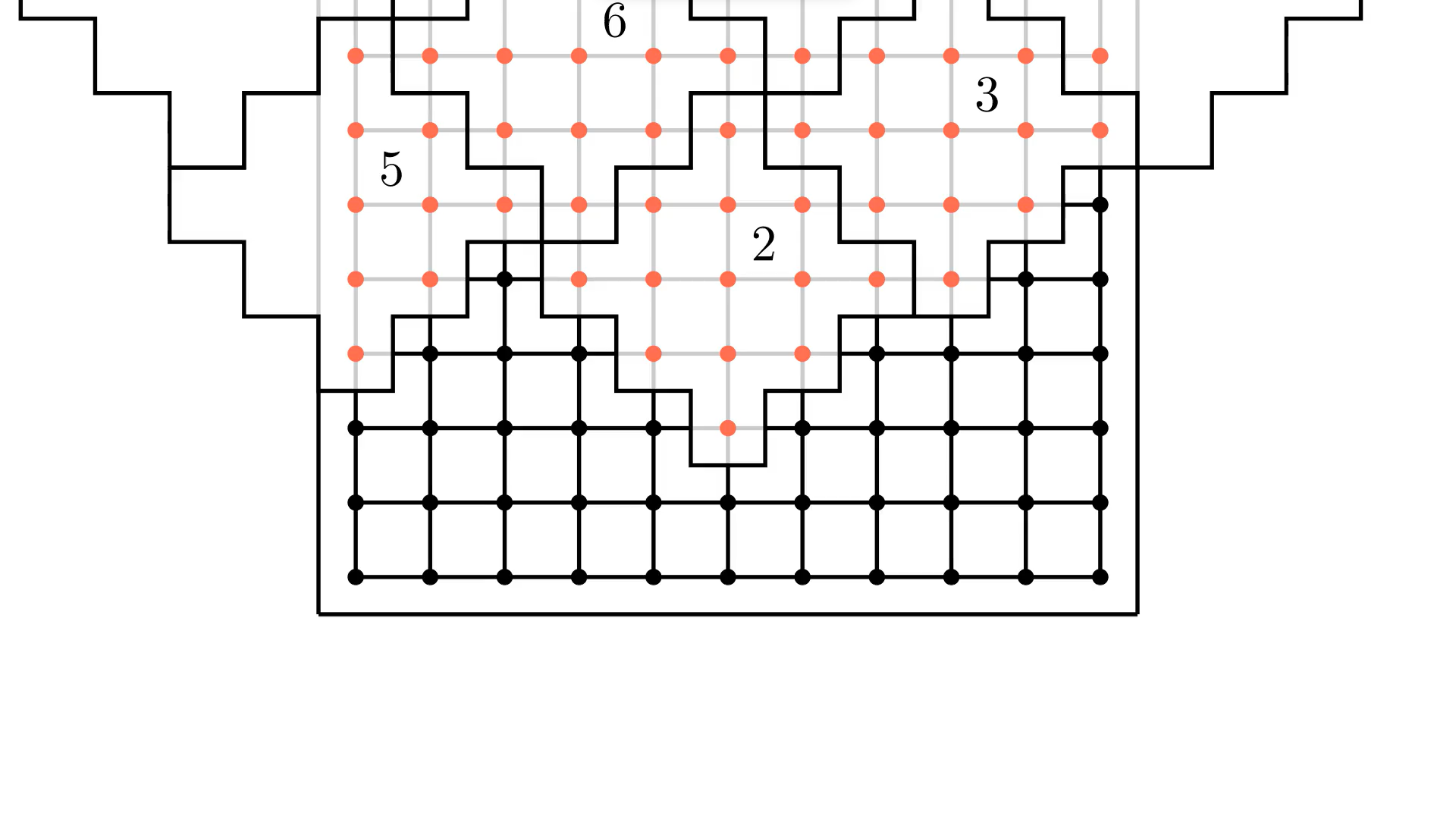}
    \includegraphics[scale=0.2]{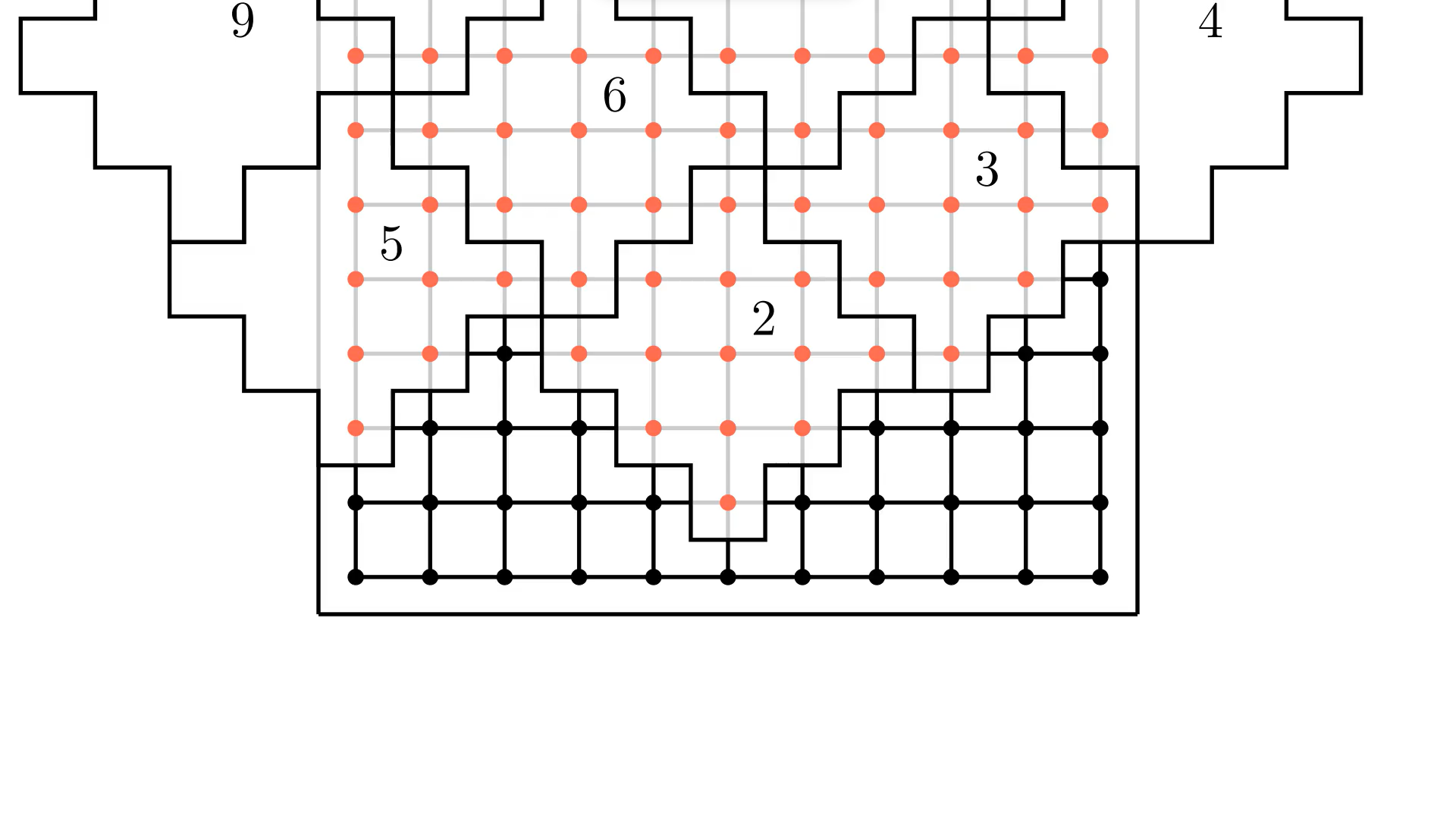}
    \includegraphics[scale=0.2]{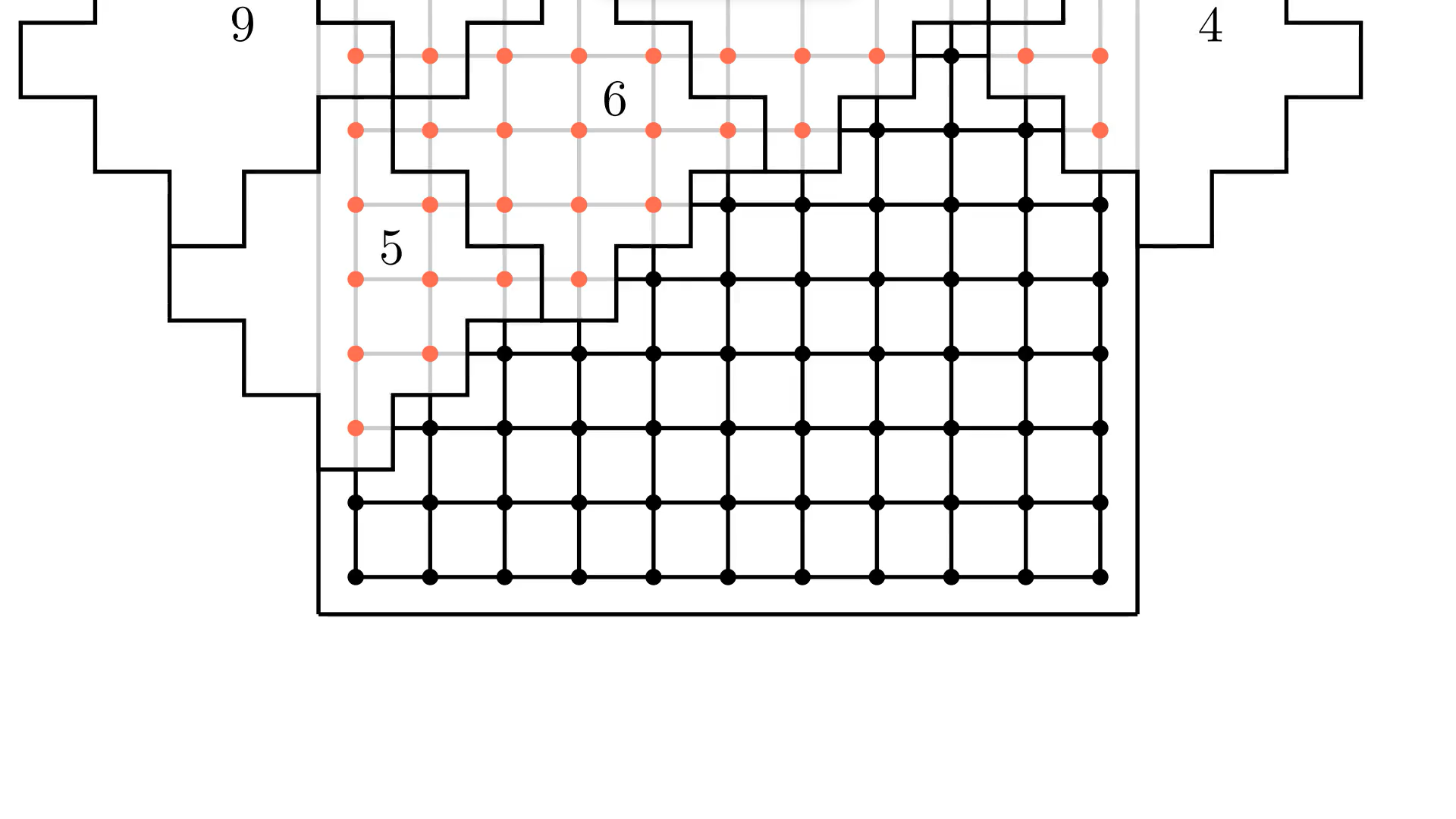}
    \caption{An example of how to control an infection on $G_{\infty,n}$. The infected vertices are red. The infection on two tiles, $T_0$ and $T_1$, is removed by two cops playing at the appropriate location (top-left to top-right images). The infection spreads to all neighboring vertices (top-right to bottom-left images). The infection on the next two tiles, $T_2$ and $T_3$, is removed by two cops playing at the appropriate location (bottom-left to bottom-right images).}
    \label{fig:rectangle_}
\end{figure}

We will consider a \emph{potentially infected} set $P_t$ of vertices, by which we mean a set of vertices containing all the vertices that the robber may be on at time $t$ but might also contain vertices that the robber cannot be on. That is, if a vertex is not in the potentially infected set, then the robber cannot be on that vertex.
It follows that if the cops can play from time $t$ and completely clear the potentially infected set $P_t$, then their strategy can capture the robber. The cop may add vertices to the current potentially infected set, which will be used to ease analysis. 

\begin{theorem} \label{thm:prox_k_grids_upper}
\[
\mathrm{prox}_k(G_{n, n}) \leq \Big\lceil \frac{n+2k+1}{2k^2+2k+1}\Big\rceil
\]    
\end{theorem}
\begin{proof}
    Recall from the previous subsection that $U_p = \bigcup_{i \geq p} T_i$ is a subset of the vertices in the rectangular grid graph $G_{\infty,n}$. 
Suppose that at some time, the potentially infected set in $G_{n, n}$ is $D^q[U_p] \cap V(G_{n, n})$. 
Note that this is true during the first step if we take $q$ to be large enough, with $[n]\times[n] \subseteq D^q[U_{0}]$ when $q \geq k \lceil \frac{n}{k+1}\rceil +k$.
Let $\hh=\Big\lceil \frac{n+2k+1}{2k^2+2k+1}\Big\rceil$ be the number of cops being used in the $k$-proximity game on $G_{n,n}$. 

On their move, the cops remove the tiles $D^q[T_p'], \ldots, D^q[T_{p+\hh-1}']$ in the $k$-proximity game of $G_{n, n}$. 
By making $\hh$ applications of Corollary~\ref{cor:cop_move_finite_grid}, this means that the robber can only reside on a vertex in $D^q[U_{p+\hh}] \cap V(G_{n, n})$. 
It is now the robber's move, and since the robber can move to a neighbor, the robber can now only reside in $N[D^q[U_{p+\hh}] \cap V(G_{n, n})]$, which is exactly $D^{q+1}[U_{p+\hh}] \cap V(G_{n, n})$ by Corollary~\ref{cor:robber_move_finite_grid}. 

The cops repeat this strategy over $t'$ rounds, and so the potentially infected set becomes $D^{q+t'}[U_{p+t'\hh}] \cap V(G_{n, n})$. 
Since we have assumed $\hh =\lceil \frac{n+2k+1}{2k^2+2k+1}\rceil$, we have that $\hh \cdot (2k^2+2k+1) > n+2k$.
Taking $t'=2k^2+2k+1$, we find that 
\begin{align*}
D^{q+2k^2+2k+1}[U_{p+(2k^2+2k+1)\hh}] 
&\subseteq D^{q+2k^2+2k+1}[U_{p+n+2k}] \\
&= D^{q+2k^2+2k+1-(2k^2+2k+1)}[U_{p}]\\
&= D^q[U_{p}],
\end{align*}
where the first subset relation follows by Corollary~\ref{cor:subsets_have_larger_index} since $\hh \cdot (2k^2+2k+1) > n+2k$ and the second relation follows by Lemma~\ref{lem:moving_down_is_moving_along}. 
That is, if the potentially infected set is $D^q[U_p] \cap V(G_{n, n})$, then after a further $t'=2k^2+2k+1$ rounds, the new potentially infected set is strictly smaller than the potentially infected set before these $t'$ rounds. 

Since there are a finite number of vertices in $G_{n, n}$, we repeat this strategy until the potentially infected set becomes empty, at which point we are certain to have captured the robber. 
\end{proof}

Surprisingly, techniques from work on the isoperimetric values for $G_{n, n}$ give a near-exact matching lower bound to the upper bound of Theorem~\ref{thm:prox_k_grids_upper}. 
\begin{theorem} \label{thm:main_grids_prox}
Let $\alpha,\beta$ be the unique values such that $n = (\alpha-1) (2k^2+2k+1) +\beta$ with $\beta \in [0,2k^2+2k]$. 
If $\beta \in [0,2k^2]$, then \[\prox_k(G_{n, n}) = \alpha.\]
Otherwise, 
\[\prox_k(G_{n, n}) \in \{\alpha,\alpha+1\}.\] 
\end{theorem}
\begin{proof}
As a consequence of Theorem~8 of \cite{MR1082842}, it was shown that $\Phi_V(G_{n,n})=n$. In the grid, there are at most $2k^2+2k+1$ vertices of distance at most $k$ from any vertex. Theorem~\ref{thm:isoperimetric} yields $\prox_k(G)\geq\lceil \frac{n+1}{2k^2+2k+1} \rceil = \alpha$. 

Theorem~\ref{thm:prox_k_grids_upper} yields $\prox_k(G_{n, n}) \leq \alpha$ when $\beta \in [0,2k^2]$, and $\prox_k(G_{n, n}) \leq \alpha+1,$ otherwise. 
\end{proof}

\subsection{The $k$-Visibility Localization Number for Cartesian Grids} \label{k_vis_localization_finite_grids}

In this subsection, we find that the values of $\zeta_k(G_{n, n})$ are determined by $\prox_k(G_{n, n})$ for most values of $k$ and $n$. Three different cases need to be considered: $k\geq 2$, $k=1$, and $k=0$. 
It will be useful to note the following.

\begin{lemma}\cite{car}
If $G$ is a graph with $\zeta_k(G)=1$, then $G$ has girth at least $6$.
\end{lemma}

The non-trivial grid graphs all contain a cycle of size 4, and so $\zeta_k(G_{n, n})>1$ unless $n=1$. We have the following corollary. 

\begin{corollary}
If $\zeta_k(G_{n, n})=1$, then $n=1$.
\end{corollary}

For the case where $k=0$, $\zeta_0(G_{n, n}) = \prox_0(G_{n, n})$ by definition of both games.   
\begin{lemma} \label{lm:zeta_prox_k0}
    For $n \geq 2$, 
    \[
    \zeta_0(G_{n, n}) = \prox_0(G_{n, n}) = n+1
    \]
\end{lemma}
\begin{proof}
Let $A_{i,j} = \{(x,y) : (x \geq i \text{ and } y \geq j) \text{ or } y \geq j+1 \},$ and let $B_{i,j}= \{(x,j) : x<i\} \cup \{(x,j-1) : x \geq i-1\}$. Note that if the robber was on $A_{i,j}$ and then moved, the robber is now on $A_{i,j-1}$. Further, note that $A_{i,j}\setminus B_{i+1,j+1}=A_{i+1,j+1},$ and that $|B_{i,j}| = n+1$.

The robber starts in $A_{1,1}$. If, at some point, the robber is on $A_{i,j}$ after its move with $i \leq n$, then the cop will probe $B_{i+1,j+1}$, meaning the robber is on $A_{i+1,j+1}$. The robber moves to a vertex in $A_{i+1,j}$.
If at some point $i \geq n+1$, then we note that $A_{n+1,j}=A_{1, j+1}$, and so we may say that the robber is on $A_{1, j+1}$, and continue the process. Over each round, the robber territory strictly decreases, so the process will terminate with the robber territory being empty after at most $n^2$ rounds. 
\end{proof}

The following lemma considers the case $k=1.$

\begin{lemma} \label{lm:grids_zeta_and_prox_k1}
Suppose $k=1$. If $\prox_1(G_{n, n}) \geq 3$, then $\zeta_1(G_{n, n})=\prox_1(G_{n, n})$.
\end{lemma}
\begin{proof}
We already have that $\zeta_1(G_{n, n}) \geq \prox_1(G_{n, n})$, by the definition of the $k$-proximity game. 
We play the $1$-visibility Localization game with $\prox_1(G_{n, n}) \geq 3$ cops and show that the robber can be captured. We begin by playing the cop strategy from the $1$-proximity game until some cop probes a distance of $1$; say the cop on vertex $(x,y)$. This is always possible since we are playing with at least $\prox_1(G_{n, n})$ cops. 
    
The robber moves. From now on, we play with only three cops. The cops probe $(x-1,y)$, $(x+1,y)$, and $(x,y+1)$. The robber is caught unless it is on:
    \begin{enumerate}
    \item $(x,y-1)$ or $(x,y-2)$, which are the vertices that  have distance at least $2$ to all cops; 
        \item $(x-2,y)$ or $(x-1,y-1)$, which are the vertices that have distance $1$ to the cop on $(x-1,y)$ and distance at least $2$ to the others; or 
        \item $(x+2,y)$ or $(x+1,y-1)$, which are the vertices that have distance $1$ to the cop on $(x+1,y)$ and distance at least $2$ to the others). 
    \end{enumerate}
In the first case, the cops play on $(x-1,y-2)$, $(x+1,y-2)$, and $(x,y-1)$, and the robber is caught.
        
In the second case (which is equivalent to the third), the robber moves, and then the cops probe $(x-2,y)$, $(x,y)$, and $(x-2,y-2)$. The robber is caught unless it is on one of the two vertices $(x-3,y)$ or $(x-2,y+1)$. The robber has been pushed down and right. 
Repeating this argument with the two original diagonal vertices switched will push the robber down and left. 
Thus, by repeating this procedure, the cops can push the robber down until at least one of the two vertices that the robber may be on is along the border of the grid. The robber will be captured using the same technique unless the two vertices the robber may be on are $(n,n-1),(n-1,n)$, or an equivalent pair of vertices in a different corner of the grid. The cops play $(n-2,n), (n,n-2), (n,n)$, and the robber is captured. 
\end{proof}

The next lemma considers when $k\ge 2.$

\begin{lemma} \label{lm:grids_zeta_and_prox}
Suppose $k\geq 2$. If $\prox_k(G_{n,n})\geq 2$, then $\zeta_k(G_{n, n})=\prox_k(G_{n, n})$.
\end{lemma}
\begin{proof}
We note that a single localization cop cannot capture a robber on an induced cycle in a graph, so it must be that $\zeta_k(G_{n, n}) \geq 2$. 
 After some time playing with $\prox_k(G_{n, n})$ cops in the Localization game, some cop will probe a finite distance to the robber. We will start our argument from this point, using only two cops on each round to capture the robber.  
    
Suppose a cop on vertex $(x,y)$ probed a finite distance, $k$, to the robber. The robber moves and is now at a distance at most $k+1$ from $(x,y)$. Note that the vertices of a fixed distance from some vertex can be decomposed into four diagonal strips. During the second cop's move, the cop places a cop on each of the vertices $(x-1,y)$ and $(x+1,y)$. By observing which of the two cops is closer to the robber, the cops can deduce if the robber's first component is less than $x$, more than $x$, or equal to $x$. 

Suppose the first case occurs (which is identical to the second case up to symmetry), meaning the robber is on a vertex in just two diagonal strips of vertices. The two cops then probe $(x-1,y-1),(x-1,y+1)$. The cops know that either that the robber is on the pair of vertices $(x,y),(x+1,y)$, which occurs if the robber was distance more than $k$ to both cops, or the robber is found to be on some single strip of points $\{(x'+i,y'+i): 1 \leq i \leq k'\}$, where $k'<k$ (or a diagonal strip of points $(x'+i,y'-i)$, but the argument is equivalent up to symmetry). 

The first subcase that the robber is on a pair of vertices $(x,y),(x+1,y)$ can be done for $y \notin \{1,n\}$ by playing a cop on $(x-1,y-1)$ and $(x-1,y+1)$. 
The robber is captured unless $k=2$ and the robber moves to a vertex in $(x+1,y),(x+2,y)$. The cops repeat inductively until the robber is against the border of the grid, at which point it is captured. In the case that $y=n$, the cops play $(x-1,n-1),(x-1,n)$, and the robber is caught immediately, and the case $y=1$ is done similarly. 
    
For the second subcase, we may assume that the robber is on the single strip of points $D=\{(x'+i,y'+i) : 1 \leq i \leq k'\}$ after the cops' second move, where $k'<k$. 
The robber moves. The cops then probe $(x',y')$ and $(x'+\lfloor k/2\rfloor,y'+\lfloor k/2\rfloor)$. 
If the first cop on $(x',y')$ probes an even distance $d$ of at most $k$ to the robber, then the robber must be on the unique vertex of $D$ of this distance to $(x',y')$, namely $(x'+d/2, y'+d/2)$. 
If the first cop probes an odd distance of at most $k$ to the robber, then the robber must be on one of the two vertices that neighbor $D$ of this distance to $(x',y')$, namely $(x'+(d+1)/2, y'+(d-1)/2)$ and $(x'+(d-1)/2, y'+(d+1)/2)$. 
If the first cop probe is $\ast$, then the robber can be assumed to have its first component more than $x'+\lfloor k/2\rfloor$ and its second component more than $y'+\lfloor k/2\rfloor$. A similar argument for the second cop shows that the robber is either captured or found to be on a pair of diagonal vertices. 
    
We can, therefore assume that the robber must be on one of two vertices $(x'',y''),(x''+1,y''+1)$, or the flip of this. 
The two cops play $(x''+1,y''-1),(x''+2,y''+1)$, and the robber will be caught unless it is on one of the two vertices $(x'',y''+1),(x''+1,y''+2)$. 
The cop repeats this strategy, making the robber move progressively upwards until the robber hits the top border of the grid and is caught.  
This completes the second subcase, and thus, the first case is complete (and by symmetry, the second case as well). 

In the remaining third case after the cops' first move, where the robber has its first component equal to $x$, the worst-case scenario for the cops is that the cops only know the robber is on one of the vertices $(x,y+k),(x,y+k+1),(x,y-k),(x,y-k-1)$. All other subcases will hold by identical logic, so we focus on this subcase only. 

The cops play on $(x,y+k-1)$ or $(x,y-k+1)$, and the robber is known to be on a set of (at most) three vertices within a common neighborhood, say (up to symmetry) $(x-1,y-k)$, $(x+1,y-k)$, or $(x,y-k-1)$. 
After the robber moves, the cops play on $(x-1,y-k)$, $(x,y-k-1)$. 
There were 11 vertices the robber could have been on: $\{(x-1,y-k+1),(x-2,y-k)\}$ have distance 1 from the first cop and distance 3 from the second cop, $\{(x,y-k),(x-1,y-k-1)\}$ have distance 1 from both cops, $\{(x+1,y-k-1),(x,y-k-2)\}$ have distance 3 from the first cop and distance 1 from the second cop, $\{(x+1,y-k+1),(x+2,y-k)\}$ have distance 3 from both cops, $(x+1,y-k)$ has distance 2 to both cops, and the other two vertices contain cops. Hence, if the robber is not caught, the robber is on two diagonal vertices, and we have already shown that the cops' have a strategy to capture the robber after such a case occurs during the second subcase of the first case. \end{proof}

Theorem~\ref{thm:main_grids_results} then follows for $k \geq 1$ by taking the values of $\prox_k(G_{n,n})$ from  Theorem~\ref{thm:main_grids_prox}, and applying either Lemma~\ref{lm:grids_zeta_and_prox_k1} or Lemma~\ref{lm:grids_zeta_and_prox} to obtain a result for $\zeta_k(G_{n,n})$. 
Lemma~\ref{lm:zeta_prox_k0} shows that 
Theorem~\ref{thm:main_grids_results} holds for $k=0$. 
    
\section{Conclusion and future directions}

We introduced the $k$-visibility localization number and proved various results, such as bounds and exact values on trees and grids. For Cartesian grids, we would like to determine which of the two values of $\zeta_k$ provided in Theorem~\ref{thm:main_grids_results} holds. Studying $\zeta_k$ in other grids, such as strong, hexagonal, triangular, or higher dimensional Cartesian grids, would also be interesting. 

While we showed that any tree may be subdivided to give $\zeta_k=1,$ this remains open for general graphs. Further, we do not know how to characterize the $k$-visibility localization number for trees. A recent paper \cite{lf} considered the localization numbers of locally finite graphs, which are countable graphs with each degree finite. An interesting direction would be considering the $\zeta_k$ numbers of locally finite trees and graphs.

\section{Acknowledgments}
The authors were supported by NSERC.

\end{document}